\documentclass[11pt,oneside]{article}
\usepackage[latin1]{inputenc}
\usepackage[T1]{fontenc}
\usepackage{amsmath}
\usepackage{amsthm}
\usepackage{amsfonts}
\usepackage{amssymb}
\usepackage{mathrsfs}
\usepackage[dvips]{graphicx}
\usepackage{float}
\usepackage[all]{xy}
\usepackage{color}
\theoremstyle{plain}
\newtheorem{prop}{Proposition}[section]
\newtheorem{teo}[prop]{Theorem}

\newtheorem*{prop3}{Proposition 3.1}
\newtheorem*{teo3}{Theorem 5.13}
\newtheorem*{teo4}{Theorem 6.2}
\newtheorem{defn}[prop]{Definition}

\newtheorem{cor}[prop]{Corollary}
\newtheorem{lem}[prop]{Lemma}
\newtheorem{es}[prop]{Example}

\newtheorem{fact}[prop]{Fact}
\theoremstyle{remark}
\newtheorem{oss}[prop]{Remark}
\definecolor{red}{rgb}{1,0,0}
\definecolor{green}{rgb}{0,1,0.2}

\title{Rational curves on $\overline{M}_g$ and K3 surfaces}
\author{L. BENZO}
\date{}
\begin{document}
\maketitle
\begin{abstract}
\noindent Let $(S,L)$ be a smooth primitively polarized K3 surface of genus $g$ and $f:X \rightarrow \mathbb{P}^1$ the fibration defined by a linear pencil in $|L|$. For $f$ general and $g \geq 7$, we work out the splitting type of the locally free sheaf $\Psi^{*}_f T_{{\overline{M}}_g}$, where $\Psi_f$ is the modular morphism associated to $f$. We show that this splitting type encodes the fundamental geometrical information attached to Mukai's projection map $\mathcal{P}_g \rightarrow \overline{\mathcal{M}}_g$, where $\mathcal{P}_g$ is the stack parameterizing pairs $(S,C)$ with $(S,L)$ as above and $C \in |L|$ a stable curve.
Moreover, we work out conditions on a fibration $f$ to induce a modular morphism $\Psi_f$ such that the normal sheaf $N_{\Psi_f}$ is locally free.
\end{abstract}
\begin{section}{Introduction}
The theory of rational curves on algebraic varieties has played a fundamental role in algebraic geometry in the last decades.\\
This has been mainly motivated by the importance of rational curves in the framework of the Minimal Model Program and of the study of higher dimensional algebraic varieties. More generally, the existence and the "number" of rational curves on a variety $X$ provides a rough measure of the complexity of $X$. For example the uniruledness of $X$, i.e. the fact that there is a rational curve passing through a general point of $X$, implies that $X$ has negative Kodaira dimension. The converse is the claim of a famous conjecture. If $X$ is not uniruled, one can consider the subvariety of $X$ spanned by rational curves and study its geometry.\\
Let $X$ be a complex projective variety of dimension $n$ and let $g:\mathbb{P}^1 \rightarrow X$ be a morphism whose image is contained in the smooth locus of $X$. The sheaf $g^{*}T_X$ is a locally free sheaf of rank $n$ over $\mathbb{P}^1$, which splits into a direct sum of line bundles by the so-called "Grothendieck's theorem".\\
Let $g^{*}T_X \cong \oplus_{i=1}^{n}\mathcal{O}_{\mathbb{P}^1}(a_i)$, $a_1 \geq a_2 \geq ... \geq a_n$. The $n$-tuple $(a_1,...,a_n)$, which we will call the \emph{splitting type} of $g^{*}T_X$, provides information concerning the deformations of the morphism $g$.\\
If $a_i \geq r$ for a nonnegative integer $r$ and all $i=1,...,n$, the morphism $g$ is said to be \emph{$r$-free}. If $g$ is $r$-free, then there is a deformation of $g$ whose image contains $r+1$ general points of the variety $X$. In particular the existence of a $0$-free ($1$-free) morphism to $X$ yields the uniruledness (rationally connectedness) of $X$ itself. $0$-free morphisms are also called \emph{free}.\\
A remarkable property is that, for a general deformation $\widetilde{g}$ of $g$, the splitting type of $\widetilde{g}^{*}T_X$ is independent on $\widetilde{g}$, a fact which will allow us to consider the splitting type associated to a fibration which is general among those considered.\\
Let now $X$ be a nonsingular projective surface and $B$ a nonsingular projective connected curve. A non-isotrivial fibration $f:X \rightarrow B$ induces a nonconstant modular map $\Psi_f:B \rightarrow \overline{M}_g$, where $\overline{M}_g$ is the coarse moduli space of stable curves of genus $g$.\\
One can consider the sheaf $N_{\Psi_f}$, which is defined as the cokernel of the map $0 \rightarrow T_B \xrightarrow{\kappa_f} Ext^1_f(\Omega^1_{X/B},\mathcal{O}_X)$. The map $\kappa_f$ is fiberwise, over the points $b$ such that the fibre $X(b)$ is reduced, the Kodaira-Spencer map of the family $f$ at $b$. If the fibration $f$ is stable and the image of $\Psi_f$ is contained in the smooth locus of $\overline{M}_g$, one has that $Ext^1_f(\Omega^1_{X/B},\mathcal{O}_X) \cong \Psi^{*}_f T_{\overline{M}_g}$ and $\kappa_f$ is exactly the inclusion in the tangent sequence, hence $N_{\Psi_f}$ can be interpreted as the normal sheaf to the the map $\Psi_f$. In the first part of the article we work out conditions on $f$ such that the sheaf $N_{\Psi_f}$ is locally free. The main statement is the following
\begin{prop3}
\label{npsiflocallyfree1}
Let $f:X \rightarrow B$ be a non-isotrivial fibration with reduced fibres such that $Ext^1_f(\Omega^1_{X/B},\mathcal{O}_X)$ is locally free and $h^{0}({T_X}_{|F})=0$ for all smooth fibres of $f$. Then $N_{\Psi_f}$ is locally free.
\end{prop3}
In the second part of the article we introduce \emph{K3-type fibrations} i.e. fibrations $f$ defined by a linear pencil $\Lambda \subset |L|$, where $L$ is an ample globally generated and primitive line bundle on a smooth K3 surface.\\
Let $\mathcal{P}_g$ be the stack parameterizing pairs $(S,C)$ such that $(S,L)$ is a smooth primitively polarized K3 surface of genus $g$ and $C \in |L|$ is a stable curve. One can consider the obvious projection $c_g:\mathcal{P}_g \rightarrow \overline{\mathcal{M}}_g$, where $\overline{\mathcal{M}}_g$ is the moduli stack of stable curves of genus $g$. Let $\mathcal{K}_g \subset \mathcal{M}_g$ be the image of $c_g$ restricted to pairs $(S,C)$ such that $C$ is smooth.\\
%$\overline{\mathcal{K}}_g \subset \overline{\mathcal{M}}_g$ is the closure in $\overline{\mathcal{M}}_g$ of the substack $\mathcal{K}_g \subset \mathcal{M}_g$ parameterizing curves lying on a K3 surface.
Our main result is the following
\begin{teo3}
\label{decomposizioneext1fk32}
Let $f$ be a general K3-type fibration of genus $g \geq 7$. Then
$$\Psi^{*}_f T_{{\overline{M}}_g} \cong \left\{\begin{array}{lc}
\mathcal{O}_{\mathbb{P}^1}(2) \oplus \mathcal{O}_{\mathbb{P}^1}(1)^{\oplus 21-g} \oplus \mathcal{O}_{\mathbb{P}^1}^{\oplus 4g-25}, & 7 \leq g \leq 9\\
\mathcal{O}_{\mathbb{P}^1}(2) \oplus \mathcal{O}_{\mathbb{P}^1}(1)^{\oplus 12} \oplus \mathcal{O}_{\mathbb{P}^1}^{\oplus 13} \oplus \mathcal{O}_{\mathbb{P}^1}(-1), & g=10\\
\mathcal{O}_{\mathbb{P}^1}(2) \oplus \mathcal{O}_{\mathbb{P}^1}(1)^{\oplus 10} \oplus \mathcal{O}_{\mathbb{P}^1}^{\oplus 19}, & g=11 \\
\mathcal{O}_{\mathbb{P}^1}(2) \oplus \mathcal{O}_{\mathbb{P}^1}(1)^{\oplus 12} \oplus \mathcal{O}_{\mathbb{P}^1}^{\oplus 17} \oplus \mathcal{O}_{\mathbb{P}^1}(-1)^{\oplus 3}, & g=12\\
\mathcal{O}_{\mathbb{P}^1}(2) \oplus \mathcal{O}_{\mathbb{P}^1}(1)^{\oplus g-1} \oplus \mathcal{O}_{\mathbb{P}^1}^{\oplus 19} \oplus \mathcal{O}_{\mathbb{P}^1}(-1)^{\oplus 2g-22}, & g \geq 13.
             \end{array}\right.$$
\end{teo3}
This can be stated also in its "geometrically interpreted" version
\begin{teo4}
Let $f$ be a general K3-type fibration of genus $g \geq 7$, let $a_g$ be the dimension of the general fibre of the morphism $c_g$, and $b_g$ the codimension of $\mathcal{K}_g$ in $\mathcal{M}_g$. Then
$$\Psi^{*}_f T_{{\overline{M}}_g} \cong \mathcal{O}_{\mathbb{P}^1}(2) \oplus \mathcal{O}_{\mathbb{P}^1}(1)^{\oplus g-1} \oplus \mathcal{O}_{\mathbb{P}^1}(1)^{\oplus a_g} \oplus \mathcal{O}_{\mathbb{P}^1}^{\oplus 2g-3-a_g-b_g} \oplus \mathcal{O}_{\mathbb{P}^1}(-1)^{\oplus b_g}.$$
\end{teo4}
The assumptions in Proposition 3.1 are satisfied by a general K3-type fibration, hence $N_{\Psi_f}$ is locally free in that case. Moreover, for $g \geq 7$, one has that
$$\Psi^{*}_f T_{\overline{M}_g} \cong T_{\mathbb{P}^1} \oplus N_{\Psi_f}.$$
\end{section}
\begin{section}{Background material}
\begin{subsection}{Basic facts about deformation theory}
In this subsection we follow the treatment in \cite{S}. By a \emph{scheme} (or $k$-scheme) we mean a locally noetherian separated scheme over an algebraically closed field $k$.\\
Let $f:X \rightarrow Y$ be a morphism of algebraic $k$-schemes, let $\mathcal{A}$ be the category of local artinian $k$-algebras with residue field $k$.
A \emph{functor of Artin rings} is defined to be a covariant functor $\mathbb{F}: \mathcal{A} \rightarrow \text{Sets}$.\\
The following functors of Artin rings will be considered in the sequel (see \cite{S}, pp. 162 and 164):\\
\begin{itemize}
\item[-] $\text{Def}_X$ (resp. $\text{Def}^{\prime}_X$), the functor of deformations (resp. locally trivial deformations) of $X$;
\item[-] $\text{Def}_f$ (resp. $\text{Def}^{\prime}_f$), the functor of deformations (resp. locally trivial deformations) of $f$;
\item[-] $\text{Def}_{f/Y}$ (resp. $\text{Def}^{\prime}_{f/Y}$), the functor of deformations (resp. locally trivial deformations) of $f$ leaving the target fixed;
\end{itemize}
If $\mathbb{F}$ is one of the functors above and $A \in \text{Ob}(\mathcal{A})$, $\mathbb{F}(A)$ will be the set of isomorphisms classes of deformations (or locally trivial deformations) of the respective data over $A$.\\
If $A=k[\epsilon]$ (the ring of the dual numbers), the deformation will be called \emph{first order}, if $A \in \text{Ob}(\mathcal{A})$ it will be called \emph{infinitesimal}.
\begin{oss}
If $Y$ is rigid as an abstract scheme (in the sequel the case $Y \cong \mathbb{P}^1$ will be extensively considered) then the functor $\text{Def}_{f/Y}$ coincides with $\text{Def}_f$. If $X$ is nonsingular then each of the functors $\text{Def}$ above coincides with the  corresponding functor $\text{Def}^{\prime}$.
\end{oss}
Let $X$ be a $k$-scheme, let $A \in \text{Ob}(\mathcal{A})$, $\lambda_0 \in \text{Pic}(X)$, $\lambda \in \text{Pic}(X \times \text{Spec}(A))$ and let $\lambda \otimes_A k$ be the pull-back of $\lambda$ under the morphism $A \rightarrow k$. It is possible to give a functor of Artin rings $\text{P}_{\lambda_0}$ by defining
$$\text{P}_{\lambda_0}(A)=\left\{\lambda \in X \times \text{Spec}(A) : \lambda \otimes_A k = \lambda_0\right\}.$$
Let $L$ be an invertible sheaf on $X$ such that $\lambda_0=[L]$. The elements of $\text{P}_{\lambda_0}(A)$ are the isomorphisms classes of deformations of $L$ over $A$.
\begin{teo}[\cite{S}, Theorem 3.3.1]
\label{xlambdaopicx}
Let $X$ be an algebraic scheme, $\lambda_0 \in \text{Pic}(X)$. If $H^{0}(X,\mathcal{O}_X) \cong k$ and $h^{1}(X,\mathcal{O}_X) < \infty$, then $$\text{P}_{\lambda_0}(k[\epsilon]) \cong H^{1}(X,\mathcal{O}_X).$$
\end{teo}
\begin{prop}[\cite{S}, Proposition 1.2.9]
Let $X$ be an algebraic variety. There is an isomorphism:
$$\kappa: \emph{Def}^{\prime}_X(k[\epsilon]) \cong H^{1}(X,T_X)$$
called the \emph{Kodaira-Spencer correspondence}.\\
In particular, if $X$ is nonsingular there is an isomorphism:
$$\kappa: \emph{Def}_X(k[\epsilon])  \cong H^{1}(X,T_X).$$
\end{prop}
\begin{defn}
For every locally trivial first order deformation $\xi$ of an algebraic variety $X$, the cohomology class $\kappa(\xi) \in H^{1}(T_X)$ is called the \emph{Kodaira-Spencer class of $\xi$.}
\end{defn}
Let $S$ be an algebraic scheme and let $\xi:\mathcal{X} \rightarrow S$ be a family of deformations of a nonsingular variety $X$. Using the 1-1 correspondence between vectors in $T_{s}S$ and morphisms $\text{Spec}(k[\epsilon]) \rightarrow S$ with image $s$, one can define a linear map
$$\kappa_{\xi}(s):T_{s}S \rightarrow H^{1}(X,T_X)$$
by sending a vector $v$ to the Kodaira-Spencer class of the first order deformation obtained by pulling back $\xi$ by the morphism corresponding to $v$.\\
The map $\kappa_{\xi}(s)$ is called the \emph{Kodaira-Spencer map of the family $\xi$ at $s$}.\\
A Kodaira-Spencer map can be defined in a more general way for reduced schemes:
\begin{prop}[\cite{S}, Theorem 2.4.1 (iv)]
\label{xreduceddefxkepsilon}
Let $X$ be a reduced algebraic scheme. There is an isomorphism:
$$\kappa: \emph{Def}_X(k[\epsilon]) \cong \emph{Ext}^1_{\mathcal{O}_X}(\Omega^1_{X},\mathcal{O}_X).$$
\end{prop}
\begin{oss}
\label{ext1conormalsequence}
The isomorphism of Proposition \ref{xreduceddefxkepsilon} is defined in the following way. Let $\xi:\mathcal{X} \rightarrow \text{Spec}(k[\epsilon])$ be a first order deformation of $X$. The conormal sequence of $X\subset \mathcal{X}$
\begin{equation}
\label{conormalsequence}
0 \rightarrow \mathcal{O}_X \rightarrow \Omega^1_{\mathcal{X}|X} \rightarrow \Omega^1_{X} \rightarrow 0
\end{equation}
is exact and yields an element of $\text{Ext}^1_{\mathcal{O}_X}(\Omega^1_X,\mathcal{O}_X)$. Define this element as $\kappa(\xi)$.\\
Given an infinitesimal deformation $\eta:\mathcal{X} \rightarrow \text{Spec}(A)$ of $X$ one has a Kodaira-Spencer map $\kappa_{\eta}:t_A \rightarrow \text{Ext}^1_{\mathcal{O}_X}(\Omega^1_X,\mathcal{O}_X)$ which associates to a tangent vector $v \in t_A$ the conormal sequence (\ref{conormalsequence}) of the pullback of $\eta$ to $\text{Spec}(k[\epsilon])$ defined by $v$.
\end{oss}
Notation as above, a case we are interested in is when $f:X \rightarrow Y$ is a closed embedding.
\begin{teo}[\cite{S}, Theorem 3.4.17]
\label{h1ty<x>}
Let $j:X \hookrightarrow Y$ be a closed embedding of projective schemes with $Y$ nonsingular and let $T_Y \langle X \rangle$ be the sheaf of tangent vectors to $Y$ which are tangent to $X$. Then $\emph{Def}^{\prime}_j$ has a formal semiuniversal deformation. Its tangent space is $H^{1}(T_Y \langle X \rangle)$ and $H^{2}(T_Y \langle X \rangle)$ is an obstruction space.\\
Moreover there is a short exact sequence
\begin{equation}
\label{ixyty<x>}
0 \rightarrow \mathcal{I}_{X/Y} \rightarrow T_Y \langle X \rangle \rightarrow T_{X} \rightarrow 0.
\end{equation}
\end{teo}
Deforming the embedding $j$ coincides with deforming the pair $(Y,X)$, thus the functor $\text{Def}^{\prime}_j$ will be also indicated with $\text{Def}^{\prime}_{(Y,X)}$.
\begin{defn}[\cite{S}, Definition 3.4.22]
If $j:X \hookrightarrow Y$ is a closed embedding of projective schemes, then $X$ is called stable in $Y$ if the (forgetful) morphism of functors of Artin rings
$$\Phi_Y:\emph{Def}'j \rightarrow \emph{Def}'Y$$
is smooth (see \cite{S}, Definition 2.4.4).
\end{defn}
Our definition of stability implies that every infinitesimal locally trivial deformation of $Y$ is induced by a locally trivial deformation of $j$.
\begin{prop}[\cite{S}, Proposition 3.4.23]
Let $j:X \hookrightarrow Y$ be a closed embedding of projective schemes with $Y$ nonsingular and let $N'_{X/Y}$ be the equisingular normal sheaf of $X$ in $Y$. If $H^{1}(X,N'_{X/Y})=(0)$, then $X$ is stable in $Y$.
\end{prop}
\begin{es}[\cite{S}, Example 3.4.24 (i)]
\label{divisorieccezionalistabili}
Let $Y$ be a projective nonsingular variety, $\gamma \subset Y$ a nonsingular closed subvariety and $\epsilon:X \rightarrow Y$ the blow-up of $Y$ with centre $\gamma$. Let $E=\epsilon^{-1}(\gamma) \subset X$ be the exceptional divisor. Then $h^{i}(E,N_{E/X})=0$ for all $i$. This means that $E$ is a stable subvariety of $X$. This is remarkable because $\gamma$ has not been required to be stable in $Y$.
\end{es}
\end{subsection}
\begin{subsection}{Rational curves on varieties}
\label{rcav}
The references for this subsection are the books of Debarre (\cite{DE}) and Kollár (\cite{KO}).\\
Let $X$ be a quasi-projective variety and $Y$ be a projective variety defined over a field $k$. There exists a locally noetherian $k$-scheme $\text{Mor}(Y,X)$ parameterizing morphisms from $Y$ to $X$ (see \cite{GR}, 4.c). This scheme will have in general countably many components.
\begin{prop}[\cite{DE}, Proposition 2.4]
\label{tmoryxf}
Notation as above, let $f:Y \rightarrow X$ be a morphism, and let $[f]$ be the corresponding point in $\emph{Mor}(Y,X)$. One has
\begin{equation}
T_{[f]}\emph{Mor}(Y,X) \cong H^{0}(Y,Hom(f^{*}\Omega^1_X,\mathcal{O}_Y)).
\end{equation}
In particular, if $X$ is smooth along the image of $f$ one has
$$T_{[f]}\emph{Mor}(Y,X) \cong H^{0}(Y,f^{*}T_X).$$
\end{prop}
\begin{teo}[\cite{DE}, Theorem 2.6]
\label{Tmoryx}
Notation as above, let $f:Y \rightarrow X$ be a morphism such that $X$ is smooth along the image of $f$. Locally around $[f]$, the scheme $\emph{Mor}(Y,X)$ can be defined by $h^{1}(Y,f^{*}T_X)$ equations in a smooth scheme of dimension $h^{0}(Y,f^{*}T_X)$. In particular:
\begin{itemize}
\item[(i)] any irreducible component of $\emph{Mor}(Y,X)$ through $[f]$ has dimension at least
$$h^{0}(Y,f^{*}T_X)-h^{1}(Y,f^{*}T_X);$$
\item[(ii)] if $h^1(Y,f^{*}T_X)=0$, then $\emph{Mor}(Y,X)$ is smooth at $[f]$.
\end{itemize}
\end{teo}
Let now $X$ be a complex
%mi serve nel teorema successivo, per non esaminare il caso di caratteristica diversa da 0 in cui ho maggiore o uguale al posto di uguale
projective variety of dimension $n$, let $f:\mathbb{P}^1 \rightarrow X$ be a morphism such that $X$ is smooth along the image of $f$ and let  $f^{*}T_X=\mathcal{O}_{\mathbb{P}^1}(a_1) \oplus ... \oplus \mathcal{O}_{\mathbb{P}^1}(a_n)$, $a_1 \geq a_2 \geq ... \geq a_n$. The numbers $a_i=a_i(f)$ are important invariants of $f$.\\
We call $\deg_{(-K_X)}(f)=\deg f^{*}T_X=\deg f^{*}(-K_X)$ the \emph{anticanonical degree of $f$}. One has $\deg_{(-K_X)}(f)=\sum_{i=1}^{n}a_i$.
\begin{prop}[\cite{KO}, II 3.12]
\label{aiindipendentidallacirr}
If  $V \subset \emph{Mor}(\mathbb{P}^1,X)$ is an irreducible component, then $\deg_{(-K_X)}(f)$ is independent of $[f] \in V$ and $a_1(f),a_2(f),...,a_n(f)$ are independent of $[f] \in V$ for $f$ general.
%basta prendere un aperto di V. Il morfismo corrispondente è contenuto nel luogo liscio di X
\end{prop}
Let $V \subset \text{Mor}(\mathbb{P}^1,X)$ be an irreducible component. Define $\text{Locus}(V)$ as the set of points $x \in X$ such that there is a morphism $f$ with $x \in f(\mathbb{P}^1)$ and $[f] \in V$. $\text{Locus}(V)$ turns out to be a (not necessarily closed) subvariety of $X$ (see \cite{KO}, II 2.3).\\
Theorem \ref{Tmoryx} is a key ingredient to prove the following
\begin{prop}[\cite{KO}, IV Theorem 2.7]
\label{dimlocusvaigeq0}
Let $X$ be a complex projective variety and $V \subset \emph{Mor}(\mathbb{P}^1,X)$ an irreducible component such that for the general $[f] \in V$ the image of $f$ is contained in the smooth locus of $X$. Let $f^{*}T_{X} \cong \mathcal{O}_{\mathbb{P}^1}(a_1) \oplus ... \oplus \mathcal{O}_{\mathbb{P}^1}(a_n)$ and suppose that $a_i \geq -1$ for all $i$. Then
\begin{equation}
\dim \emph{ Locus}(V) = \#\left\{i | a_i \geq 0\right\}.
\end{equation}
\end{prop}
%l'enunciato di Kollár dice Mor_bir invece che Mor, ma Mor_bir è un aperto di Mor sempre per Kollár, Definition 2.6
\end{subsection}
\begin{subsection}{Fibrations}
From now on we will assume $k = \mathbb{C}$.\\
By a \emph{fibration} we mean a surjective morphism with connected fibres $f:X \rightarrow B$ from a projective nonsingular surface to a projective nonsingular connected curve.\\
A fibration is said to be
\begin{itemize}
\item \emph{relatively minimal} if there are no (-1)-curves contained in any of its fibres.
\item \emph{semistable} if it is relatively minimal and every fibre has at most nodes as singularities.
\item \emph{stable} if every fibre has at most nodes as singularities and finite automorphisms group.
\item \emph{non-isotrivial} if all its nonsingular fibres are mutually isomorphic.
\end{itemize}
We will denote by $g$ the genus of a general fibre and by $b$ the genus of the base curve $B$. We will always assume $g \geq 2$.\\
Sheaves of differentials will be denoted with the symbol $\Omega$ and dualizing sheaves with the symbol $\omega$. Let $f:X \rightarrow B$ be any fibration. Since both $X$ and $B$ are nonsingular one has $\wedge^2 \Omega^1_X=\omega_X$ and $\Omega^1_B=\omega_B$.\\
Since $f$ is a relative complete intersection morphism we have the equality (see \cite{K}, Corollary 24)
\begin{equation}
\label{relativecanonicalsheaf}
\omega_{X/B}=\omega_{X} \otimes f^{*}{\omega_B}^{-1}.
\end{equation}
There exists an exact sequence
\begin{equation}
\label{f*omegabomega1x}
0 \rightarrow f^{*}\omega_{B} \rightarrow \Omega^1_{X} \rightarrow \Omega^1_{X/B} \rightarrow 0
\end{equation}
(which is exact on the left because the first homomorphism is injective on a dense open set and $f^{*}\omega_B$ is locally free).\\
Define $T_{X/B} \doteq  Hom(\Omega^1_{X/B},\mathcal{O}_X)$. The following is a classical result of Arakelov in the semistable case, and it is due to Serrano in the general case (recall that we are assuming $g \geq 2$):
\begin{teo}
\label{as}
If $f$ is a non-isotrivial fibration, then $h^0(X, T_{X/B})=h^{0}(X,T_X)=0$.
If moreover $f$ is relatively minimal, then one also has that $h^1(X, T_{X/B})=0$.
\end{teo}
\begin{proof}
$f$ is non-isotrivial if and only if $f_ {*}{T_X}=0$ (\cite{SE2}, Lemma 3.2). Since $f_{*}T_{X/B} \subset f_{*}T_X$ one also has $f_{*}T_{X/B}=0$, and the first equality immediately follows. For the second equality, see \cite{A} or \cite{SZ} in the semistable case and \cite{SE}, Corollary 3.6, in the general case.
\end{proof}
Let $f:X \rightarrow B$ be a non-isotrivial fibration. Consider the exact sequence (\ref{f*omegabomega1x}) and apply on it the left exact functor $f_{*}Hom$. Since the sheaves $f_{*}Hom(\Omega^{1}_{X/B},\mathcal{O}_X)$ and $f_{*}Hom(\Omega^1_X,\mathcal{O}_X)$ equal $0$ by the proof of Theorem \ref{as}, one obtains the 4-terms exact sequence
$$0 \rightarrow f_{*}Hom(f^{*}\omega_{B},\mathcal{O}_X) \rightarrow Ext^1_f(\Omega^1_{X/B},\mathcal{O}_X) \rightarrow Ext^1_f(\Omega^{1}_X,\mathcal{O}_X) \rightarrow$$
\begin{equation}
\label{successionef*hom}
\rightarrow Ext^1_f(f^{*}\omega_{B},\mathcal{O}_X) \rightarrow 0
\end{equation}
where $Ext^i_f(\cdot,\mathcal{O}_X)$ is the $i$-th derived functor of the left exact functor $f_{*} \circ Hom (\cdot,\mathcal{O}_X)$ and the sheaves $Ext^2_f$ are all $0$ because the fibres of $f$ are 1-dimensional.
Let us compute each of the terms in the sequence.\\
Since $\omega_B$ is invertible, $Hom$ and pull-back commute and so $f_{*}Hom(f^{*}\omega_{B},\mathcal{O}_X)=f_{*}f^{*}(T_{B})=T_{B}$ by the projection formula.\\
By \cite{L}, (3) p. 102 and the projection formula we have $Ext^1_f(\Omega^{1}_X,\mathcal{O}_X) \cong R^1 f_{*}(T_X)$ and $Ext^1_f(f^{*}\omega_{B},\mathcal{O}_X)=R^1f_{*}(f^{*}T_{B})=T_{B} \otimes R^1f_{*}\mathcal{O}_X$, hence sequence (\ref{successionef*hom}) rewrites as
\begin{equation}
\label{tbext1fomega1xb}
0 \rightarrow T_B \xrightarrow{\kappa_f} Ext^1_f(\Omega^1_{X/B},\mathcal{O}_X) \rightarrow R^1 f_{*}T_X \rightarrow T_B \otimes R^1 f_{*}\mathcal{O}_X \rightarrow 0.
\end{equation}
We are interested in the sheaf $Ext^1_f(\Omega^1_{X/B},\mathcal{O}_X)$ because its cohomology describes the deformation theory of $f$ (see Lemma \ref{significatoh0h1ext1f}). This sheaf is not locally free in general, but it decomposes as follows:
\begin{equation}
\label{ext1ftorsionell}
Ext^1_f(\Omega^1_{X/B},\mathcal{O}_X)=\mathcal{E} \oplus \mathcal{T}
\end{equation}
where $\mathcal{E}$ is locally free and $\mathcal{T}$ is a torsion sheaf. By the \emph{rank} of $Ext^1_f(\Omega^1_{X/B},\mathcal{O}_X)$ we mean the rank of $\mathcal{E}$.
\begin{lem}
\label{relativedualityExt1f}
For any fibration $f : X \rightarrow B$ the sheaf $Ext^1_f(\Omega^1_{X/B},\mathcal{O}_X)$ has rank $3g-3$.
Moreover, there is an exact sequence of sheaves on $B$
\begin{equation}
\label{r1f*txbc10}
0 \rightarrow  R^1f_{*}T_{X/B} \xrightarrow{c_{10}} Ext^1_f(\Omega^1_{X/B},\mathcal{O}_X) \xrightarrow{c_{01}} f_{*}Ext^1_{\mathcal{O}_X}(\Omega^1_{X/B},\mathcal{O}_X) \rightarrow 0.
\end{equation}
%Moreover, there is an exact sequence of sheaves on $B$
%\begin{equation}
%\label{r1f*txbc10}
%0 \rightarrow  R^1f_{*}T_{X/B} \xrightarrow{c_{10}} Ext^1_f(\Omega^1_{X/B},\mathcal{O}_X) \xrightarrow{c_{01}} f_{*}N \rightarrow 0.
%\end{equation}
%If $Ext^1_f(\Omega^1_{X/B},\mathcal{O}_X)$ is locally free and all the fibres of $f$ are reduced, then there is an isomorphism
If $f$ is stable, then there is an isomorphism
\begin{equation}
\label{ext1fhomf*omega1xbomegaxb}
Ext^1_f(\Omega^1_{X/B},\mathcal{O}_X) \cong Hom(f_{*}(\Omega^1_{X/B} \otimes \omega_{X/B}),\mathcal{O}_B)
\end{equation}
so that in particular $Ext^1_f(\Omega^1_{X/B},\mathcal{O}_X)$ is locally free.
\end{lem}
\begin{proof}
Let $b \in B$. If $X(b)$ is smooth, then $\text{Ext}^1_{\mathcal{O}_{X(b)}}(\Omega^1_{X(b)},\mathcal{O}_{X(b)}) \cong H^1(T_{X(b)})$ has dimension $3g-3$.
Then, if $U \subset B$ is the open set over which $f$ is smooth, one has that $Ext^1_f(\Omega^1_{X/B},\mathcal{O}_X)_{|U}$ is locally free of rank $3g-3$.\\
(\ref{r1f*txbc10}) is the sequence associated to the local-to-global spectral sequence for $Ext_f$.\\
If $f$ is stable $\Omega^1_{X/B}$ is torsion-free, hence flat over $B$, and the sheaf $f_{*}(\Omega^1_{X/B} \otimes \omega_{X/B})$ commutes with base change. As a consequence, relative duality (see \cite{K}, Corollary 24) can be applied and (\ref{ext1fhomf*omega1xbomegaxb}) follows.
\end{proof}

\begin{prop}[\cite{S4}, Proposition 1.7]
If the fibration $f$ is non-isotrivial one has
\begin{equation}
\label{chiext1fomega1xb}
\chi(Ext^1_f(\Omega^1_{X/B},\mathcal{O}_X))=11\chi(\mathcal{O}_X)-2K^2_X+2(b-1)(g-1).
\end{equation}
\end{prop}
$\left.\right.$\\
If we define $N_{\Psi_f}$ as the cokernel of the homomorphism $\kappa_f$ in sequence (\ref{tbext1fomega1xb}), we obtain two short exact sequences:
\begin{equation}
\label{tbkfext1fnpsif}
0 \rightarrow T_B \xrightarrow{\kappa_f} Ext^1_f(\Omega^1_{X/B},\mathcal{O}_X) \rightarrow N_{\Psi_f} \rightarrow 0
\end{equation}
\begin{equation}
\label{npsifr1f*txtb}
0 \rightarrow N_{\Psi_f} \rightarrow R^1 f_{*}T_X \rightarrow T_B \otimes R^1 f_{*}\mathcal{O}_X \rightarrow 0.
\end{equation}
In case $B \cong \mathbb{P}^1$, from the cohomology sequence of (\ref{tbkfext1fnpsif}) one obtains immediately
\begin{equation}
\label{relationcohomologyext1npsif}
h^{0}(N_{\Psi_f})=h^{0}(Ext^1_f(\Omega^1_{X/\mathbb{P}^1},\mathcal{O}_X))-3, \quad h^1(N_{\Psi_f})=h^{1}(Ext^1_f(\Omega^1_{X/\mathbb{P}^1},\mathcal{O}_X)).
\end{equation}
Analogously, from the same cohomology sequence tensorized by $\mathcal{O}_{\mathbb{P}^1}(-1)$ it follows that
\begin{equation}
\label{npsif-1ext1f-1}
h^1(N_{\Psi_f}(-1))=h^{1}(Ext^1_f(\Omega^1_{X/\mathbb{P}^1},\mathcal{O}_X)(-1)).
\end{equation}
$\left.\right.$\\
Assume now that the fibration $f$ is also stable. To $f$ we can associate a modular morphism
$$\Psi_f: B \rightarrow \overline{\mathcal{M}}_g$$
to the \emph{moduli stack} of stable curves of genus $g$. We know (\cite{HMU}, p. 49) that $f_{*}(\Omega^1_{X/B} \otimes \omega_{X/B}) \cong \Psi^{*}_f\Omega^1_{\overline{\mathcal{M}}_g}$, thus Lemma \ref{relativedualityExt1f} gives
\begin{equation}
\label{ext1ftmg}
Ext^1_f(\Omega^1_{X/B},\mathcal{O}_X) \cong \Psi^{*}_f T_{\overline{\mathcal{M}}_g}.
\end{equation}
From (\ref{ext1ftmg}) we are then allowed to interpret $N_{\Psi_f}$ as the \emph{normal sheaf to the moduli map $\Psi_f$}.\\ \\
If the fibration $f$ is not stable, we still have a non-empty open set $U \subset B$ above which all fibres of $f$ are stable. Therefore we have an induced morphism $U \rightarrow \overline{\mathcal{M}}_g$; since $B$ is a nonsingular curve, this morphism extends to a morphism
$$\Psi_f:B \rightarrow \overline{M}_g$$
with values in the \emph{coarse moduli space}.\\
Now, assume the image of $\Psi_f$ to be contained in the smooth locus of $\overline{M}_g$. Then $\Psi^{*}_f T_{\overline{M}_g}$ is a locally free sheaf. If $f$ is stable, this sheaf is isomorphic to $\Psi^{*}_f T_{\overline{\mathcal{M}}_g}$, and we can think $N_{\Psi_f}$ in sequences (\ref{tbkfext1fnpsif}) and (\ref{npsifr1f*txtb}) as the normal sheaf to the map $\Psi_f$ with values in the coarse moduli space.
Since we want to use the results contained in Subsection \ref{rcav}, which refer to rational curves on \emph{algebraic varieties}, we will always consider the map $\Psi_f$ to have values in the coarse moduli space.\\
%In the sequel we will extensively refer to $N_{\Psi_f}$ without making the assumption that $\Psi_f$ is contained in the smooth locus of $\overline{M}_g$. In these cases $N_{\Psi_f}$ is intended to be defined simply as the cokernel of the map $\kappa_f$ in sequence (\ref{tbkfext1fnpsif}).
%In the sequel, with a little abuse of notation, we will use the symbol $N_{\Psi_f}$ also for fibrations which are not necessarily semistable. The reason is that in many of the situation involving $N_{\Psi_f}$ which we are going to examine, a general deformation of $f$ will be semistable. If one focuses on the non-semistable case, $N_{\Psi_f}$ is intended to refer simply to the kernel of the homomorphism $\kappa_f$, without any geometrical interpretation.
\end{subsection}
\begin{subsection}{Deformations of fibrations and rational fibrations}
\begin{lem}[\cite{S4}, Lemma 2.1]
\label{significatoh0h1ext1f}
Let $f:X \rightarrow B$ be a non-isotrivial fibration. Then there is a natural isomorphism
$$\mu:\emph{Ext}^1_{\mathcal{O}_X}(\Omega^1_{X/B},\mathcal{O}_X) \rightarrow H^{0}(Ext^1_f(\Omega^1_{X/B},\mathcal{O}_X))$$
and both spaces are naturally identified with the tangent space of $\emph{Def}_{f/B}$ , the functor of Artin rings of deformations of $f$ leaving the target fixed.\\
Moreover $H^{1}(Ext^1_f(\Omega^1_{X/B},\mathcal{O}_X))$ is an obstruction space for $\emph{Def}_{f/B}$.
\end{lem}
\begin{prop}[\cite{S4}, Lemma 2.2 (ii)]
\label{mappasullefibreèkodairaspencer}
Let $f:X \rightarrow B$ be a non-isotrivial fibration. For a point $b \in B$ such that the fibre $X(b)$ is reduced, the homomorphism $\kappa_f:T_B \rightarrow Ext^1_f(\Omega^1_{X/B},\mathcal{O}_X)$ in the exact sequence (\ref{tbext1fomega1xb}) induces a linear map:
\begin{equation}
\label{kfp}
\kappa_f(b):T_{b}B \rightarrow \emph{Ext}^1_{\mathcal{O}_{X(b)}}(\Omega^1_{X(b)},\mathcal{O}_{X(b)})
\end{equation}
which coincides with the Kodaira-Spencer map of the family $f$ at $b$.\\
\end{prop}
We will call \emph{rational fibrations} the fibrations parametrized by $\mathbb{P}^1$. All examples of rational fibrations are obtained as follows.\\
Let $S$ be a projective nonsingular surface and let $C \subset S$ be a projective nonsingular connected curve of genus $g$ such that
$\dim(|C|) \geq 1$.
Consider a linear pencil $\Lambda$ contained in $|C|$ whose general member is nonsingular and let $\epsilon:X \rightarrow S$ be the blow-up at the (possibly empty) set of its base points (including the infinitely near ones). We obtain a fibration $f:X \rightarrow \mathbb{P}^1$
defined as the composition of $\epsilon$ with the rational map $S \dashrightarrow \mathbb{P}^1$ defined by $\Lambda$. We will call $f$ the \emph{fibration defined by the pencil $\Lambda$}.
\begin{defn}[\cite{S4}, Definition 3.1]
Let $f:X \rightarrow \mathbb{P}^1$ be a fibration of genus $g$ and let $\mathcal{E}$ be as in \emph{(\ref{ext1ftorsionell})}. Then one has
\begin{equation}
\label{ext1fcongai}
\mathcal{E} \cong \bigoplus_{i=1}^{3g-3}\mathcal{O}_{\mathbb{P}^1}(a_i)
\end{equation}
for some integers $a_i$, $a_1 \geq a_2 \geq ... \geq a_{3g-3}$. $f$ is said to be \emph{free} if it is non-isotrivial and $Ext^1_f(\Omega^1_{X/\mathbb{P}^1},\mathcal{O}_X)$ is globally generated i.e. $a_i \geq 0$ for all $i$.
\end{defn}
The $(3g-3)$-tuple $(a_1,...,a_{3g-3})$ will be called the \emph{splitting type} of $Ext^1_f(\Omega^1_{X/\mathbb{P}^1},\mathcal{O}_X)$.\\
The following proposition explains the relation between curves with general moduli and free fibrations:
\begin{prop}[\cite{S4}, Proposition 4.8 and Theorem 4.9]
\label{relazionecurvageneraleffree}
Assume that $C$ is a general (nonsingular) curve of genus $g \geq 3$ moving in a positive-dimensional linear system on a projective nonsingular non-ruled surface $S$ (see \emph{\cite{S}}, Definition 4.3). Then a general pencil $\Lambda \subset |C|$ containing $C$ as a member defines a free fibration. On the other hand, if there exists a free fibration $f:X \rightarrow \mathbb{P}^1$ in curves of genus $g \geq 3$, then a general deformation of $f$ has $C$ as a fibre.
\end{prop}
The following proposition relates the cohomology of the twisted sheaves $Ext^1_f(\Omega^1_{X/\mathbb{P}^1},\mathcal{O}_X)(k)$ with the cohomology of $T_X(kF)$, which is easier to compute. First, let us write down explicit-ly the low-degree terms exact sequence associated to the Leray spectral sequence for an arbitrary sheaf of abelian groups $\mathscr{F}$ on $X$ (see \cite{CE}, cap. XV, Theorem 5.11):
$$0 \rightarrow H^{1}(B,f_{*}\mathscr{F}) \rightarrow H^{1}(X,\mathscr{F}) \rightarrow H^{0}(B,R^1f_{*}\mathscr{F}) \rightarrow$$
\begin{equation}
\label{ssf}
\rightarrow H^2(B,f_{*}\mathscr{F}) \rightarrow H^2(X,\mathscr{F}) \rightarrow H^{1}(B,R^1f_{*}\mathscr{F}) \rightarrow 0.
\end{equation}
\begin{prop}
\label{unsoloai=2}
Let $f:X \rightarrow \mathbb{P}^1$ be a non-isotrivial fibration in curves of genus $g$, let $\mathcal{E}$ as in \emph{(\ref{ext1fcongai})}, and let $F$ be a fibre of $f$. Then
\begin{itemize}
\item[(i)] for every integer $k$ one has $h^{0}(N_{\Psi_f}(k)) \leq h^{1}(T_X(kF))$;
\item[(ii)] if $X$ is regular, $h^{0}(N_{\Psi_f}(-2)) = h^{1}(T_X(-2F))$; in particular if $h^1(T_X(-2F))=0$ one has $a_1=2$ and $a_i \leq 1$ for $i=2,...,3g-3$.
\end{itemize}
\end{prop}
\begin{proof}
Consider the exact sequence
\begin{equation}
\label{sequenzanormale2tensorizzata}
0 \rightarrow N_{\Psi_f}(k) \rightarrow R^1 f_{*}T_X(kF) \rightarrow R^1 f_{*}\mathcal{O}_X \otimes \mathcal{O}_{\mathbb{P}^1}(2+k) \rightarrow 0
\end{equation}
obtained by tensorizing sequence (\ref{npsifr1f*txtb}) by $\mathcal{O}_{\mathbb{P}^1}(k)$.\\
Since $f_{*}T_X(kF)=f_{*}T_X \otimes \mathcal{O}_{\mathbb{P}^1}(k)$ equals 0 by the proof of Theorem \ref{as}, sequence (\ref{ssf}) gives $h^{1}(T_X(kF))=h^{0}(R^1 f_{*}T_X(kF))$, hence from the cohomology sequence associated to (\ref{sequenzanormale2tensorizzata}) one gets $h^{0}(N_{\Psi_f}(k)) \leq h^{1}(T_X(kF))$.\\
Let $k=-2$. Again by (\ref{ssf}) and by the regularity of $X$ one has $0=h^{1}(\mathcal{O}_X)=h^{0}(R^1 f_{*}\mathcal{O}_X)$, from which $h^{0}(N_{\Psi_f}(-2)) = h^{1}(T_X(-2F))$.\\
If $h^{1}(T_X(-2F))=0$, one obtains $h^{0}(Ext^1_f(\Omega^1_{X/\mathbb{P}^1},\mathcal{O}_X)(-2))=1$ from the cohomology sequence of the short exact sequence (\ref{tbkfext1fnpsif}) tensorized by $\mathcal{O}_{\mathbb{P}^1}(-2)$. Since $f$ is non-isotrivial, (\ref{tbkfext1fnpsif}) gives $a_1 \geq 2$. The last part of $(ii)$ immediately follows.
\end{proof}
Observe that equality (\ref{relationcohomologyext1npsif}) and $(i)$ with $k=0$ tell that the dimension of the space of first order deformations of $f$ is bounded above by $h^{1}(T_X)+3$.
\end{subsection}
\end{section}
\begin{section}{The Kodaira-Spencer map of a fibration}
\begin{prop}
\label{npsiflocallyfree}
Let $f:X \rightarrow B$ be a non-isotrivial fibration with reduced fibres such that $Ext^1_f(\Omega^1_{X/B},\mathcal{O}_X)$ is locally free and $h^{0}({T_X}_{|F})=0$ for all smooth fibres of $f$. Then $N_{\Psi_f}$ is locally free.
\end{prop}
\begin{proof}
By the exact sequence (\ref{tbkfext1fnpsif}), the fact that $N_{\Psi_f}$ is locally free is equivalent to the injectivity (i.e. to the non degeneracy), for all $b \in B$, of the evaluation map $\kappa_f(b)$, which by Proposition \ref{mappasullefibreèkodairaspencer} is nothing but the Kodaira-Spencer map of the family $f$ at $b$.\\
Let $F=X(b)$ be a singular fibre of $f$. Since $F$ is reduced, by Remark \ref{ext1conormalsequence} $\kappa_f(b)$ is injective if and only if the sequence
%h^{1}(T_S(_C)) è lo spazio tangente alla fibra
\begin{equation}
\label{ofomega1xfmoega1f}
0 \rightarrow \mathcal{O}_F \rightarrow \Omega^1_{\mathcal{F}|F} \rightarrow \Omega^1_{F} \rightarrow 0
\end{equation}
does not split for some first order deformation $\mathcal{F}$ of $F$ obtained as the pull-back over $\text{Spec}(k[\epsilon])$ of the deformation
\begin{equation}
\xymatrix{F \ar[r] \ar[d] & X \ar[d]^f \\ b \ar[r] & B}
\end{equation}
by some vector $v \in T_{b}B$.\\
This is in turn implied by the nonsplitting of the sequence
\begin{equation}
\label{sequenzaconormaleF}
0 \rightarrow \mathcal{O}_F \rightarrow \Omega^1_{X|F} \rightarrow \Omega^1_{F} \rightarrow 0
\end{equation}
which is obvious since $\Omega^1_{F}$ is not locally free (see \cite{H}, II Theorem 8.15).\\
%ci sono sempre deformazioni al primo ordine che lisciano una fibra singolare
%oppure, se fosse localmente libero anche il suo duale T_F lo sarebbe
Let now $F=X(b)$ be a smooth fibre and consider the exact sequence
$$0 \rightarrow T_F \rightarrow {T_{X}}_{|F} \rightarrow N_{F/X} \rightarrow 0.$$
The vector space $H^{0}(N_{F/X})$ is the tangent space at $[F]$ to $\text{Hilb}_X(F) \cong B$, hence it can be identified with the tangent space $T_{b}B$ and the first coboundary map of the associated cohomology sequence is exactly the Kodaira-Spencer map of the family $f$ at $b$:
$$H^{0}({T_{X}}_{|F}) \rightarrow H^{0}(N_{F/X}) \xrightarrow{\kappa_f(b)} H^{1}(T_F).$$
Since $H^{0}({T_X}_{|F})=(0)$, the map $\kappa_f(b)$ must be injective. This completes the proof.
\end{proof}
\begin{fact}
If $f$ is a stable fibration, then $h^{0}({T_X}_{|F})=0$ for all singular fibres $F$.
\end{fact}
\begin{proof}
Let $F$ be a singular fibre of $f$ and consider the exact sequence
\begin{equation}
\label{tftxfn'fx}
0 \rightarrow T_F \rightarrow {T_X}_{|F} \rightarrow N'_{F/X} \rightarrow 0
\end{equation}
where $N'_{F/X}$ is the equisingular normal sheaf of $F$ in $X$.\\
Since the fibration is stable, $h^{0}(T_F)=0$ (cf. \cite{H2}, p. 181). Moreover, since $N'_{F/X} \cong \mathcal{I}_{p_1,...,p_{\delta}} \otimes N_{F/X}$ where $p_1,...,p_{\delta}$ are the nodes of $F$, one has $h^{0}(N'_{F/X})=0$. The cohomology sequence of (\ref{tftxfn'fx}) gives the statement.
\end{proof}
As a consequence, if the fibration is stable and $h^{0}({T_X}_{|F})=0$ for all smooth fibres of $f$, then $h^{0}({T_X}_{|F})=0$ for all fibres of $f$. In particular, this value does not depend on the fibre.\\
In the particular case $B \cong \mathbb{P}^1$, the linear equivalence of the fibres can be used to prove the following useful:
\begin{prop}
\label{R1f*TXlocallyfree}
Let $f:X \rightarrow \mathbb{P}^1$ be a non-isotrivial fibration such that $h^{0}({T_X}_{|F})$ does not depend on the fibre. Then $R^1f_{*}T_X$ is locally free. In particular $N_{\Psi_f}$ is locally free.
\end{prop}
\begin{proof}
Since $T_X$ is $f$-flat, we can apply \cite{H}, III Corollary 12.9, which tells that $R^1f_{*}T_X$ is locally free if $h^{1}({T_X}_{|F})$ does not depend on the fibre. Since all fibres are linearly equivalent, the cohomology sequence associated to the short exact sequence
$$0 \rightarrow T_X(-F) \rightarrow T_X \rightarrow {T_X}_{|F} \rightarrow 0$$
shows that $h^{0}({T_X}_{|F})-h^{1}({T_X}_{|F})$ does not depend on the fibre, thus $R^1f_{*}T_X$ is locally free. By sequence (\ref{npsifr1f*txtb}), $N_{\Psi_f}$ is locally free too.
\end{proof}
It is easy to construct an example of a fibration $f'$ such that $N_{\Psi_{f'}}$ is not locally free. Consider a non-isotrivial stable rational fibration $f$ and the diagram
\begin{equation}
\xymatrix{X' \ar[r] \ar[d]_{f'} & X \ar[d]^{f} \\\mathbb{P}^1 \ar[r]^{g}  & \mathbb{P}^1}
\end{equation}
where $g$ is a double covering of $\mathbb{P}^1$ and $X'$ is the fiber product. Suppose that $g$ is branched over points $q_1,q_2$ such that $ X(q_i)$ is smooth, so that $X'$ is smooth.\\
Let $\Psi_{f'}:\mathbb{P}^1 \rightarrow \overline{M}_g$ be the modular morphism associated to $f'$. One has $\text{d}\Psi_{f'}(q_i)=\kappa_{f'}(q_i)$, where $\kappa_{f'}(q_i)$ is the Kodaira-Spencer map of the family $f'$ at $q_i$. On the other hand $\Psi_{f'}=\Psi_f \circ g$ by construction, hence $\text{d} \Psi_{f'}(q_i)$ is degenerate. By the first part of the proof of Proposition \ref{npsiflocallyfree}, $N_{\Psi_{f'}}$ cannot be locally free at $q_i$.
\begin{prop}
Let $f:X \rightarrow B$ be a non-isotrivial fibration, let $F \doteq X(b)$ be a smooth fibre and let $\kappa_f(b):H^{0}(N_{F/X}) \rightarrow H^{1}(T_F)$ be the Kodaira-Spencer map. Then one of the following occurs:
\begin{itemize}
\item[(i)] $h^{0}({T_X}_{|F})=0$, and this holds if and only if $\kappa_f(b)$ is injective;
\item[(ii)] $h^{0}({T_X}_{|F})=1$, and this holds if and only if $\kappa_f(b)$ has image zero.
\end{itemize}
Moreover, (i) holds if and only if $im(\alpha) \cap im(\xi)=(0)$ in the following diagram with exact rows and columns
\begin{equation}
\xymatrix{ & H^{1}(T_F) \ar[r]^{=} & H^{1}(T_F) \\ 0 \ar[r] & H^{0}(N_{F/X}) \ar[u]^{\kappa_f(b)} \ar[r]^{\alpha} & H^{1}(T_X \langle F \rangle) \ar[u]^{p} \ar[r]^{\beta} & H^1(T_X) \ar[d]^{=} \\0 \ar[r] & H^{0}({T_X}_{|F}) \ar[u] \ar[r] & H^{1}(T_X(-F)) \ar[u]^{\xi} \ar[r]^{\gamma} & H^1(T_X) \\ & 0 \ar[u] & 0 \ar[u] &}
\end{equation}
\end{prop}
\begin{proof}
The first part of the statement is obvious since $N_{F/X} \cong \mathcal{O}_F$.\\
%il normale è per forza quello, perchè ha grado 0 e ha sezioni
By the proof of Theorem \ref{as} one has $h^{0}(T_X)=0$, thus $\kappa_f(b)$ is injective if and only if $\gamma$ is injective. Since the diagram is commutative, $\gamma$ is injective if and only if $\beta \circ \xi$ is, i.e. if and only if $im(\alpha) \cap im(\xi)=(0)$.
%But this is the case if and only if first order deformations of $F$ in $X$ are not all trivial.
\end{proof}
Note that, by construction, $p$ is the map sending a (isomorphism class of a) first order deformation of the pair $(X,F)$ to the corresponding first order deformation of $F$. Hence the image of $\xi$ in $H^{1}(T_X \langle F \rangle)$ is the subspace parameterizing isomorphisms classes of first order deformations of $(X,F)$ inducing trivial first order deformations of $F$.\\
The map $\alpha$ sends a first order deformation $F'$ of $F$ into the isomorphism class $(X,F')$.\\
\end{section}
\begin{section}{Functors of Artin rings associated to a rational fibration defined by a linear pencil}
\label{Functorsofartinrings}
Let $S$ be a smooth projective surface and let $C \subset S$ be a connected nonsingular curve moving in a positive-dimensional linear system. Let $f:X \rightarrow \mathbb{P}^1$ be the rational fibration defined by a linear pencil in $|C|$.\\
Let $A \in \text{Ob}(\mathcal{A})$, let $C_1,C_2 \in |C|$ and let $h_i:C_i \hookrightarrow S$ be their closed embedding. A \emph{deformation of the 3-tuple $(S,C_1,C_2)$ over $A$} is defined to be a cartesian diagram
\begin{equation}
\xymatrix{(C_1,C_2) \ar[d]^{(h_1,h_2)} \ar[r] & (\mathcal{C}_1,\mathcal{C}_2) \ar[d]^{(H_1,H_2)}\\S \ar[r] \ar[d] & \mathcal{S} \ar[d]^{\Psi} \\ \text{Spec}(k) \ar[r] & \text{Spec}(A)}
\end{equation}
where $\Psi$ and $\Psi \circ H_i$ are flat.\\
Define the functor of Artin rings $\hbox{Def}^{\prime}_{(S,C_1,C_2)}$ as
$$\hbox{Def}^{\prime}_{(S,C_1,C_2)}(A)=\left\{\begin{array}{c}
                                        \hbox{isomorphisms classes of locally trivial} \\
                                        \hbox{deformations of $(S,C_1,C_2)$ over }A
                                      \end{array}\right\}.$$
This functor admits a semiuniversal formal deformation. The proof of this fact is identical to the one given in the case of pairs in \cite{S}, Theorem 3.4.17. In particular, there is a well-defined tangent space to the functor, the space $\hbox{Def}^{\prime}_{(S,C_1,C_2)}(k[\epsilon])$.\\
Let $F_1,F_2$ be two fibres of $f$ and let $j_i:F_i \rightarrow X$, $i=1,2$ their closed embedding.
In an analogous way one can define the functor of Artin rings $\text{Def}^{\prime}_{(X,F_1,F_2)}$ as
$$\hbox{Def}^{\prime}_{(X,F_1,F_2)}(A)=\left\{\begin{array}{c}
                                        \hbox{isomorphisms classes of locally trivial} \\
                                        \hbox{deformations of $(X,F_1,F_2)$ over }A
                                      \end{array}\right\}.$$
Let $A \in \hbox{ob}(\mathcal{A})$. A \emph{deformation of the 3-tuple $(f,F_1,F_2)$ over $A$} is defined to be a cartesian diagram
\begin{equation}
\xymatrix{(F_1,F_2) \ar[d]^{(j_1,j_2)} \ar[r] & (\mathcal{F}_1,\mathcal{F}_2) \ar[d]^{(J_1,J_2)}\\X \ar[r] \ar[d]^{f} & \mathcal{X} \ar[d]^{F} \\\mathbb{P}^1 \ar[r] \ar[d] & \mathbb{P}^1 \times \text{Spec}(A) \ar[d]^{\Psi}\\ \text{Spec}(k) \ar[r] & \text{Spec}(A)}
\end{equation}
where $\Psi$, $\Psi \circ F$ and $F \circ J_j$ are flat.\\
Define the functor of Artin rings $\hbox{Def}^{\prime}_{(f,F_1,F_2)}$ as
$$\hbox{Def}^{\prime}_{(f,F_1,F_2)}(A)=\left\{\begin{array}{c}
                                        \hbox{isomorphisms classes of locally trivial} \\
                                        \hbox{deformations of $(f,F_1,F_2)$ over }A
                                      \end{array}\right\}.$$
In order to prove that $\hbox{Def}^{\prime}_{(f,F_1,F_2)}(k[\epsilon])$ has a structure of vector space one has to prove that the functor $\text{Def}^{\prime}_{(f,F_1,F_2)}$ satisfies conditions $H_0$ and $H_{\epsilon}$ of Schlessinger's theorem (see \cite{S}, Theorem 2.3.2).\\
Condition $H_0$ is obviously satisfied. From the fact that the functors $\text{Def}^{\prime}_{F_i}$, $i=1,2$ and $\text{Def}^{\prime}_f$ satisfy condition  $H_{\epsilon}$ (see \cite{S}, Corollary 2.4.2 and Theorem 3.4.8), it immediately follows that also $\hbox{Def}^{\prime}_{(f,F_1,F_2)}$ has to.
%vedi tesi8
\begin{lem}
\label{defff1f2defxf1f2}
Notation as above, let $X$ be a regular surface. Then one has $\emph{Def}^{\prime}_{(f,F_1,F_2)}(k[\epsilon]) \cong \emph{Def}^{\prime}_{(X,F_1,F_2)}(k[\epsilon]) \oplus H^{0}(f^{*}T_{\mathbb{P}^1})$.
\end{lem}
\begin{proof}
Consider the forgetful map $\beta: \hbox{Def}^{\prime}_{(f,F_1,F_2)} \rightarrow \hbox{Def}^{\prime}_{(X,F_1,F_2)}$ and take $\lambda_{0}=\mathcal{O}_X(F_1)=\mathcal{O}_X(F_2) \in \hbox{Pic}(X)$. Since $X$ is regular, by Theorem \ref{xlambdaopicx} one has $P_{\lambda_{0}}(k[\epsilon])=H^{1}(X,\mathcal{O}_X)=(0)$, thus for every first order deformation
\begin{equation}
\label{rappresentantexia}
\xymatrix{(F_1,F_2) \ar[d]^{(j_1,j_2)} \ar[r] & (\mathcal{F}_1,\mathcal{F}_2) \ar[d]^{(J_1,J_2)}\\X \ar[r] \ar[d] & \mathcal{X} \ar[d]^{\Psi} \\ \text{Spec}(k) \ar[r] & \text{Spec}(k[\epsilon])}
\end{equation}
of $(X,F_1,F_2)$ one has that $\mathcal{O}_{\mathcal{X}}(\mathcal{F}_1)=\mathcal{O}_{\mathcal{X}}(\mathcal{F}_2)$ i.e. $\mathcal{F}_1$ and $\mathcal{F}_2$ are linearly equivalent in $\hbox{Div}(\mathcal{X})$. Consequently, (\ref{rappresentantexia}) is the image of some first order deformation of $(f,F_1,F_2)$ i.e. the differential $d\beta:\hbox{Def}^{\prime}_{(f,F_1,F_2)}(k[\epsilon]) \rightarrow \hbox{Def}^{\prime}_{(X,F_1,F_2)}(k[\epsilon])$ is surjective. The kernel of $d \beta$ parametrizes isomorphisms classes of locally trivial first order deformations of $f$ leaving the domain, the target and the two fibres $F_1,F_2$ fixed, thus by \cite{S}, Proposition 3.4.2, it is isomorphic to $H^{0}(f^{*}T_{\mathbb{P}^1})$.
\end{proof}
\begin{teo}
\label{equivalencexf1xf2sc1c2}
Notation as above, let $S$ be a regular surface, let $F_i$ be the strict transform of $C_i$, $i=1,2$, in the blow-up $\epsilon:X \rightarrow S$ and let $C_1 \neq C_2$. Then there is an equivalence of functors of Artin rings between $\emph{Def}^{\prime}_{(X,F_1,F_2)}$ and $\emph{Def}^{\prime}_{(S,C_1,C_2)}$.
\end{teo}
\begin{proof}
Consider the map of functors of Artin rings $\gamma: \hbox{Def}^{\prime}_{(X,F_1,F_2)} \rightarrow \hbox{Def}^{\prime}_{(S,C_1,C_2)}$ defined in the following way. Let $A \in \hbox{ob}(\mathcal{A})$, let $\xi_{A} \in \hbox{Def}^{\prime}_{(X,F_1,F_2)}(A)$ having as representative a deformation
\begin{equation}
\label{rappresentantexikepsilon}
\xymatrix{(F_1,F_2) \ar[d]^{(j_1,j_2)} \ar[r] & (\mathcal{F}_1,\mathcal{F}_2) \ar[d]^{(J_1,J_2)}\\X \ar[r] \ar[d] & \mathcal{X} \ar[d] \\ \text{Spec}(k) \ar[r] & \text{Spec}(A)}
\end{equation}
By Example \ref{divisorieccezionalistabili} the exceptional divisors $E_1,...,E_k$ of the blow-up $\epsilon$ are stable, thus for all $i=1,...,k$ every infinitesimal locally trivial deformation of $X$ is induced by an infinitesimal locally trivial deformation of the embedding $e_i:E_i \hookrightarrow X$.\\
%vedi definizione di funtore liscio, S def 2.4.4.
Hence (\ref{rappresentantexikepsilon}) induces a deformation of the morphism $\epsilon$
\begin{equation}
\xymatrix{X \ar[r] \ar[d]^{\epsilon} & \mathcal{X} \ar[d]^{\mathcal{E}} \\S \ar[r] \ar[d] & \mathcal{S} \ar[d] \\ \text{Spec}(k) \ar[r] & \text{Spec}(A)}
\end{equation}
which is in particular a deformation of $S$.\\
One defines $\gamma_A(\xi_A)$ as the element having as representative the deformation
\begin{equation}
\xymatrix{(\epsilon(F_1),\epsilon(F_2)) \ar[d]^{(h_1,h_2)} \ar[r] & (\mathcal{E}(\mathcal{F}_1),\mathcal{E}(\mathcal{F}_2)) \ar[d]^{(H_1,H_2)}\\S \ar[r] \ar[d] & \mathcal{S} \ar[d] \\ \text{Spec}(k) \ar[r] & \text{Spec}(A)}
\end{equation}
where $h_i: \epsilon(F_i) \hookrightarrow S$, $i=1,2$, is the embedding.\\
For all morphisms $A \rightarrow B$, $A,B \in \hbox{ob}(\mathcal{A})$ one has that the diagram
\begin{equation}
\xymatrix{\hbox{Def}^{\prime}_{(X,F_1,F_2)}(A) \ar[d] \ar[r]^{\gamma_A} & \hbox{Def}^{\prime}_{(S,C_1,C_2)}(A) \ar[d] \\
\hbox{Def}^{\prime}_{(X,F_1,F_2)}(B) \ar[r]^{\gamma_B} & \hbox{Def}^{\prime}_{(S,C_1,C_2)}(B)}
\end{equation}
is commutative, hence $\gamma$ is a well-defined map of functors.\\
Let us define an inverse for $\gamma$, say $\gamma^{\prime}$. Let $\mu_A \in \hbox{Def}^{\prime}_{(S,C_1,C_2)}(A)$ having as representative the diagram
\begin{equation}
\xymatrix{(C_1,C_2) \ar[d]^{(h_1,h_2)} \ar[r] & (\mathcal{C}_1,\mathcal{C}_2) \ar[d]^{(H_1,H_2)}\\S \ar[r] \ar[d] & \mathcal{S} \ar[d] \\ \text{Spec}(k) \ar[r] & \text{Spec}(A)}
\end{equation}
and let $Z\doteq \left\{p_1,...,p_{k}\right\}$ be the zero-dimensional subscheme of $S$ whose points are the base points of the pencil spanned by $C_1$ and $C_2$. Since $S$ is regular, $\mathcal{C}_1$ and $\mathcal{C}_2$ are linearly equivalent in $\mathcal{S}$, thus $\mu_A$ induces a deformation of $(S,Z)$ over $A$ giving a deformation of $\epsilon:X \rightarrow S$ over $A$
%i punti base della deformazione Z' di Z sono ancora punti base di un pencil in |C|
\begin{equation}
\xymatrix{X \ar[r] \ar[d]^{\epsilon} & \mathcal{X} \ar[d]^{\mathcal{E}} \\S \ar[r] \ar[d] & \mathcal{S} \ar[d] \\ \text{Spec}(k) \ar[r] & \text{Spec}(A)}
\end{equation}
%è piatta
One defines $\gamma^{\prime}(\mu_A)$ as the element having as representative the deformation
\begin{equation}
\xymatrix{(\widetilde{C}_1,\widetilde{C}_2) \ar[d]^{(j_1,j_2)} \ar[r] & (\widetilde{\mathcal{C}}_1,\widetilde{\mathcal{C}}_2) \ar[d]^{(J_1,J_2)}\\X \ar[r] \ar[d]^{f} & \mathcal{X} \ar[d] \\ \text{Spec}(k) \ar[r] & \text{Spec}(A)}
\end{equation}
where $\widetilde{C}_i$ is the strict transform of $C_i$ and $j_i:\widetilde{C}_i \hookrightarrow X$, $i=1,2$ is the embedding.\\
$\gamma'$ is a well-defined map of functors and by construction $\gamma$ and $\gamma^{\prime}$ are inverse of each other, hence there is a natural equivalence between $\hbox{Def}^{\prime}_{(X,F_1,F_2)}$ and $\hbox{Def}^{\prime}_{(S,C_1,C_2)}$.
\end{proof}
\end{section}
\begin{section}{The main theorem}
Consider a primitive globally generated and ample line bundle $L$ on a smooth K3 surface $S$. Then the general member of $|L|$ is a smooth irreducible curve (see \cite{SA}). If $g$ is the genus of that curve,
we say that $(S,L)$ is a \emph{smooth primitively polarized K3 surface of genus $g$}. One has $L^2=2g-2$ and $\dim|L|=g$. In the sequel we will always assume $g \geq 3$.\\
If $\text{Pic}(S) \cong \mathbb{Z}[L]$, then $L$ is very ample on $S$, hence it embeds $S$ as a smooth projective surface in $\mathbb{P}^g$, whose smooth hyperplane section is a canonical curve (see \cite{SA}).\\
Let $\mathcal{F}_g$ be the moduli stack parameterizing smooth primitively polarized K3 surfaces $(S,L)$ of genus $g$. $\mathcal{F}_g$ is smooth, irreducible and of dimension $19$ (see e.g. \cite{BPV}, VIII Theorem 7.3). Moreover, for a general pair $(S,L) \in \mathcal{F}_g$ one has $\text{Pic}(S) \cong \mathbb{Z}[L]$.\\
Let $\mathcal{P}_g$ be the moduli stack parameterizing pairs $(S,C)$, where $S$ is a smooth K3 surface with a primitive polarization $L$ of genus $g$ and $C \in |L|$ is a stable curve of genus $g$. We have a surjective morphism of stacks
$$\pi_g:\mathcal{P}_g \rightarrow \mathcal{F}_g$$
given by the natural projection.\\
Since the fibres of $\pi_g$ are irreducible, $\mathcal{P}_g$ is an irreducible stack of dimension $19+g$. The tangent space to $\mathcal{P}_g$ at a pair $(S,C)$ with $C$ smooth is $H^{1}(T_S \langle C \rangle)$ (see e.g. \cite{FLA}, Section 3).\\
We have a morphism of stacks
$$c_g:\mathcal{P}_g \rightarrow \overline{\mathcal{M}}_g$$
defined as $c_g((S,C))=[C]$. Define $\mathcal{K}_g$ as the image of the morphism $c_g$ restricted to pairs $(S,C)$ such that $C$ is smooth, and $\overline{\mathcal{K}}_g$ as the closure of $\mathcal{K}_g$ in $\overline{\mathcal{M}}_g$.\\
Since $19+g \geq 3g-3$ if and only if $g \leq 11$, one could expect naively $c_g$ to be dominant exactly for these values of $g$. Actually the situation is not that simple:
\begin{teo}
\label{cg}
With notation as above one has
\begin{equation}
\label{codimkg}
\text{codim}_{\mathcal{M}_g}(\mathcal{K}_g)=\left\{\begin{array}{cc}
                                0, & g \leq 9 \hbox{ or } g=11 \\
                                1, & g=10 \\
                                3, & g=12 \\
                                2g-22, & g \geq 13.
                                \end{array}\right.
\end{equation}
In particular, $c_g$ is dominant for $g \leq 9$ and $g = 11$ (cf. \emph{\cite{M1}}), but not for $g=10$ (cf. \emph{\cite{CU}}), and generically finite for $g=11$ and $g \geq 13$, but not for $g=12$ (cf. \emph{\cite{M2}}).
\end{teo}
If we denote by $F_g$ the general fibre of the morphism $c_g$, it follows that $\dim(F_g)$ is the expected one apart from the cases $g=10,12$.\\
Consider a linear pencil $\Lambda \subset |L|$ whose general element is nonsingular ($\Lambda$ is always non-isotrivial by \cite{CK}, Proposition 1.2). Let $\epsilon:X \rightarrow S$ be the blow-up at the base-points $p_1,...,p_{2g-2}$ of $\Lambda$. The rational fibration  $f:X \rightarrow \mathbb{P}^1$ in curves of genus $g$ ($\geq 3$) defined by $\Lambda$ will be called a \emph{K3-type fibration (of genus $g$)}.\\
By a \emph{general K3-type fibration} we mean a K3-type fibration $f$ such that $(S,L)$ is a general element of $\mathcal{F}_g$ and moreover the pencil $\Lambda$ is general as a line in $|L| \cong \mathbb{P}^g$.\\
In this case $\Lambda$ is a Lefschetz pencil i.e. every singular curve in $\Lambda$ is irreducible with one node and no other singularities, thus $f$ is a stable fibration. In particular, by Lemma \ref{relativedualityExt1f} one has that $Ext^1_f(\Omega^1_{X/\mathbb{P}^1},\mathcal{O}_X)$ is locally free.
%\begin{oss}
%Let $f$ be a general K3-type fibration. Then $f$ is stable.
%\end{oss}
%\begin{proof}
%Every smooth curve of genus $g \geq 2$ is stable. Moreover the fibres of $f$ are hyperplane sections with respect to the projective embedding %given by $L$, hence by Lefschetz theory every singular fibre of $f$ is irreducible with one node and no other singularities, thus stable.
%\end{proof}
\begin{prop}
\label{npsiflocallyfreek3}
Let $f$ be a K3-type fibration defined by a linear pencil whose elements are irreducible curves with at most ordinary nodes. Then $N_{\Psi_f}$ is locally free.
\end{prop}
\begin{proof}
Since $(S,L)$ is a primitive polarization, all fibres of $f$ are reduced. By Proposition \ref{npsiflocallyfree} it is then sufficient to show that $h^{0}({T_X}_{|F})=0$ for any smooth fibre $F$. Consider the inclusion of sheaves $0 \rightarrow T_X \hookrightarrow \epsilon^{*}T_S$ and restrict it to the fibre $F$, which is the strict transform of a curve $C \subset S$ which is smooth or nodal. Applying the left exact functor $\epsilon_{*}$ we obtain $h^{0}({T_X}_{|F}) \leq h^{0}(\epsilon_{*}(\epsilon^{*}{T_S}_{|F}))=h^{0}({T_S}_{|C})$. Since $h^{0}({T_S}_{|C})=0$ by the proof of \cite{CK}, Proposition 1.2, $h^{0}({T_X}_{|F})$ must be zero.
%potrei avere un pencil generato da due curve aventi una cuspide in un punto base. Il risultato di Ciliberto vale solo per curve nodali.
\end{proof}
In particular this holds for a general K3-type fibration.
\begin{prop}
\label{h0h1ext1f}
Let $f:X \rightarrow \mathbb{P}^1$ be a K3-type fibration of genus $g$. Then one has $h^{0}(Ext^1_f(\Omega^1_{X/\mathbb{P}^1},\mathcal{O}_X))=2g+20$ and $h^{1}(Ext^1_f(\Omega^1_{X/\mathbb{P}^1},\mathcal{O}_X))=0$. In particular deformations of $f$ are unobstructed.
\end{prop}
\begin{proof}
Notation as in Theorem \ref{equivalencexf1xf2sc1c2}, let $(S,L)$ be a primitively polarized K3 surface of genus $g$ and let $C_1,C_2$ be general elements in $\Lambda$, $C_1 \neq C_2$. Let $\alpha: \hbox{Def}^{\prime}_{(S,C_1,C_2)} \rightarrow \hbox{Def}^{\prime}_{(S,C_1)}$ be the forgetful map and consider its differential $d \alpha:\hbox{Def}^{\prime}_{(S,C_1,C_2)}(k[\epsilon]) \rightarrow \hbox{Def}^{\prime}_{(S,C_1)}(k[\epsilon])$, which is surjective. The kernel of $d \alpha$ is the vector space parameterizing isomorphisms classes of locally trivial first order deformations of $(S,C_1,C_2)$ which induce trivial first order deformations of the pair $(S,C_1)$. Hence its dimension is the dimension of the image of the Kodaira-Spencer map $\kappa: H^{0}(N_{{C_2}/S}) \rightarrow H^{1}(T_{C_2})$, which is $g$ since $\kappa$ is injective by \cite{CK}, Proposition 1.2.\\
%C_2 è liscia, si potrebbe fare anche con nodata, ma in ogni caso la proposizione non vale per ogni fibrazione di tipo K3 perché il risultato di Ciliberto sull'iniettività della mappa di KS vale per curve al più nodate.
Since $\hbox{Def}^{\prime}_{(S,C_1)}(k[\epsilon])$ is isomorphic to $H^1(T_S \langle C_1 \rangle)$ (see Theorem \ref{h1ty<x>}), which has dimension $19+g$, one obtains that $\hbox{Def}^{\prime}_{(S,C_1,C_2)}(k[\epsilon])$ has dimension $19+2g$, which is also the dimension of $\hbox{Def}^{\prime}_{(X,F_1,F_2)}(k[\epsilon])$ by Theorem \ref{equivalencexf1xf2sc1c2}.\\
By Lemma \ref{defff1f2defxf1f2} one has that $\hbox{Def}^{\prime}_{(f,F_1,F_2)}(k[\epsilon])$ has dimension $19+2g+3=22+2g$.\\
In the end consider the forgetful map $\eta:\hbox{Def}^{\prime}_{(f,F_1,F_2)} \rightarrow \hbox{Def}^{\prime}_{f}$. The kernel of  the (surjective) differential $d \eta$ is isomorphic to the image of the Kodaira-Spencer map $\kappa^{\prime}: H^{0}(N_{F_1/X}) \oplus H^{0}(N_{F_2/X}) \rightarrow H^{1}(T_{F_1}) \oplus H^{1}(T_{F_2})$, which has dimension 2. It follows that $\dim \hbox{Def}^{\prime}_{f}(k[\epsilon])=h^{0}(Ext^1_f(\Omega^1_{X/\mathbb{P}^1},\mathcal{O}_X))=2g+20$. By equality (\ref{chiext1fomega1xb}) one has $h^{1}(Ext^1_f(\Omega^1_{X/\mathbb{P}^1},\mathcal{O}_X))=0$. The unobstructedness follows from Lemma \ref{significatoh0h1ext1f}.
\end{proof}
\begin{prop}
\label{generalk3typefibrationfreeiff}
Let $f$ be a general K3-type fibration of genus $g$. Then $f$ is free if and only if $g \leq 9$ or $g=11$.
\end{prop}
\begin{proof}
By construction, the stability of the $E_i$ (see Example \ref{divisorieccezionalistabili}) and Proposition \ref{h0h1ext1f}, a (small) deformation of a K3-type fibration $f$ of genus $g$ is again a K3-type fibration (of the same genus). Using Theorem \ref{cg} one can then conclude that there are no deformations of $f$ having as fibre the general curve of genus $g$ for $g=10$ or $g \geq 12$, hence in these cases $f$ is not free by Proposition \ref{relazionecurvageneraleffree}.\\
On the other hand, if $g \leq 9$ or $g=11$ then the linear system $|L|$ of a general smooth primitively polarized K3 surface $(S,L)$ of genus $g$ contains the general curve of genus $g$, again by Theorem \ref{cg}. Thus the general pencil $\Lambda \subset |L|$ defines a free fibration by Proposition \ref{relazionecurvageneraleffree}.
\end{proof}
Let $f$ be a K3-type fibration of genus $g$. Consider the short exact sequence (see (\ref{npsifr1f*txtb}))
\begin{equation}
\label{sequenzafondamentale}
0 \rightarrow N_{\Psi_f} \rightarrow R^1 f_{*}T_X \rightarrow T_{\mathbb{P}^1} \otimes R^1 f_{*} \mathcal{O}_X \rightarrow 0.
\end{equation}
Let $F$ be a fibre of $f$. We have that $h^1(F,\mathcal{O}_F)=g$ and thus by \cite{H}, III Corollary 12.9, $R^1 f_{*} \mathcal{O}_X$ is a locally free sheaf of rank $g$. From sequence (\ref{ssf}) one obtains the two equalities $0=h^{1}(\mathcal{O}_X)=h^{0}(R^1 f_{*}\mathcal{O}_X)$ and
$1=h^{2}(\mathcal{O}_X)=h^{1}(R^1 f_{*} \mathcal{O}_X)$.\\
Then one can conclude that $R^1 f_{*} \mathcal{O}_X=\mathcal{O}_{\mathbb{P}^1}(-2) \oplus \mathcal{O}_{\mathbb{P}^1}(-1)^{\oplus g-1}$, from which
\begin{equation}
\label{tp1tensorr1fstarox}
T_{\mathbb{P}^1} \otimes R^1 f_{*} \mathcal{O}_X=\mathcal{O}_{\mathbb{P}^1} \oplus \mathcal{O}_{\mathbb{P}^1}(1)^{\oplus g-1}.
\end{equation}
By Serre and Hodge duality one has that $h^2(T_S)=h^{0}(\Omega^1_S)=h^1(\mathcal{O}_S)=0$ and $h^{0}(T_S)=h^2(\Omega^1_S)=h^1(\omega_S)=0$.\\
Since $f$ is non-isotrivial, one has $h^{0}(T_X)=0$ by Theorem \ref{as}. Moreover one has that $h^{2}(T_X)=h^{2}(T_S)$.\\
By Riemann-Roch formula one has $\chi(T_X)=2K_{X}^2-10 \chi(\mathcal{O}_X)$ and then $h^1(T_X)=16+4g$.
By sequence (\ref{ssf}) one has $h^{0}(R^1 f_{*}T_X)=16+4g$ and $h^{1}(R^1 f_{*}T_X)=0$. The cohomology sequence associated to (\ref{sequenzafondamentale}) is then
$$0 \rightarrow H^{0}(N_{\Psi_f}) \rightarrow H^{0}(R^1 f_{*} T_X) \rightarrow H^{0}(T_{\mathbb{P}^1} \otimes R^1 f_{*} \mathcal{O}_X) \rightarrow H^{1}(N_{\Psi_f}) \rightarrow 0.$$
One can compute $\chi(N_{\Psi_f})=2g+17$ using this cohomology sequence or equality (\ref{chiext1fomega1xb}) combined with equality (\ref{relationcohomologyext1npsif}).\\
The following lemmas lead us to compute the cohomology of $T_X(-F)$ and $T_X(-2F)$, which is related to our problem by Proposition \ref{unsoloai=2}.
\begin{lem}[\cite{CLM}, Lemma 4]
\label{h0ncpgk3}
Let $C$ be a general hyperplane section of a general primitively polarized K3 surface of genus $g \geq 6$. If $k \geq 3$ then $H^0(N_{C/\mathbb{P}^{g-1}}(-k)) =(0)$. If $g \geq 7$ then $H^0(N_{C/\mathbb{P}^{g-1}}(-2)) =(0)$; if $g = 6$ then $h^{0}(N_{C/\mathbb{P}^{g-1}}(-2)) \leq 1$.
\end{lem}
\begin{lem}[\cite{CD}, Lemma 2.4]
\label{conoxc}
Let $C$ be a reduced and irreducible, not necessarily smooth, degenerate canonical curve in $\mathbb{P}^g$, of arithmetic genus $g$. Let $X_C$ be the cone over $C$ from a point in $\mathbb{P}^g$ off the hyperplane in which $C$ sits. For all $i \geq 0$, one has
$$H^{0}(N_{X_C/\mathbb{P}^g}(-i)) \cong \bigoplus_{k \geq i} H^{0}(N_{C/\mathbb{P}^{g-1}}(-k)).$$
\end{lem}
\begin{lem}
\label{h1ts-ck3}
Let $(S,L)$ be a general primitively polarized K3 surface of genus $g$ and $C$ be a general curve in $|L|$. Then one has
\begin{equation}
h^1(T_S(-C))=\left\{\begin{array}{cc}
                      22-2g, & g \leq 9\\
                      3, & g=10 \\
                      0, & g=11 \hbox{ or } g \geq 13 \\
                      1, & g=12.
                    \end{array}\right.
\end{equation}
and
$h^1(T_S(-2C))=0$ for all $g \geq 7$.
\end{lem}
\begin{proof}
The cohomology sequence associated to the exact sequence $0 \rightarrow T_S(-C) \rightarrow T_S \langle C \rangle \rightarrow T_C \rightarrow 0$ writes as
$$0 \rightarrow H^{1}(T_S(-C)) \rightarrow H^{1}(T_S \langle C \rangle) \xrightarrow{\text{d} c_g} H^1(T_C) \rightarrow H^{2}(T_S(-C)) \rightarrow 0$$
where $\text{d}c_g$ is the tangent map to the morphism $c_g$ at $(S,C)$ (see \cite{B1}, Section 5). Since the fibre $F_g$ of $c_g$ over $[C]$ is reduced, one has $h^{1}(T_S(-C))=\dim F_g$. Theorem \ref{cg} then gives the first part of the statement.\\
Let $g \geq 7$. By the cohomology sequence associated to the exact sequence
$$0 \rightarrow T_S(-2) \rightarrow {T_{\mathbb{P}^{g}}}_{|S}(-2) \rightarrow N_{S/\mathbb{P}^{g}}(-2) \rightarrow 0$$
it is sufficient to show that $h^{0}(N_{S/\mathbb{P}^{g}}(-2))=h^{1}({T_{\mathbb{P}^{g}}}_{|S}(-2))=0$.\\
Let $X_C$ be as in Lemma \ref{conoxc}. Since $S$ is projectively Cohen-Macaulay, it flatly degene-rates to $X_C$ (see \cite{CD}, Lemma 2.3), thus one has $h^{0}(N_{S/\mathbb{P}^g}(-2)) \leq h^{0}(N_{X_C/\mathbb{P}^g}(-2))$ by upper semicontinuity. On the other hand $h^{0}(N_{X_C/\mathbb{P}^g}(-2))=0$ by Lemma \ref{conoxc} and Lemma \ref{h0ncpgk3}.\\
One has $H^{0}({T_{\mathbb{P}^{g}}}_{|S}(-2))=(0)$ by the Euler sequence restricted to $S$ and tensorized by $\mathcal{O}_S(-2)$. The same sequence tensorized by $\mathcal{O}_S(-1)$ gives
$$0 \rightarrow H^{0}(\mathcal{O}_S(-1)) \rightarrow H^{0}(\mathcal{O}_S)^{\oplus g+1} \rightarrow H^{0}(T_{\mathbb{P}^g|S}(-1)) \rightarrow H^1(\mathcal{O}_S(-1)) \rightarrow $$
\begin{equation}
\label{h0os-1h0osg+1}
\rightarrow H^1(\mathcal{O}_S)^{\oplus g+1} \rightarrow H^1(T_{\mathbb{P}^g|S}(-1)) \rightarrow H^2(\mathcal{O}_S(-1)) \xrightarrow{a} H^2(\mathcal{O}_S)^{\oplus g+1}.
\end{equation}
Now, $h^{1}(\mathcal{O}_S)=0$ and the map $a$ is injective by duality and the projective normality of $S$ (see \cite{MA}, Proposition 2), thus $h^1({T_{\mathbb{P}^{g}}}_{|S}(-1))=0$.\\
The short exact sequence $0 \rightarrow {T_{\mathbb{P}^{g}}}_{|S}(-2) \rightarrow {T_{\mathbb{P}^{g}}}_{|S}(-1) \rightarrow {T_{\mathbb{P}^{g}}}_{|C}(-1) \rightarrow 0$ then gives
$$0 \rightarrow H^{0}({T_{\mathbb{P}^{g}}}_{|S}(-1)) \rightarrow H^{0}({T_{\mathbb{P}^{g}}}_{|C}(-1)) \rightarrow H^{1}({T_{\mathbb{P}^{g}}}_{|S}(-2)) \rightarrow 0$$
and one has that $h^{1}({T_{\mathbb{P}^{g}}}_{|S}(-2))=0$ if and only if $H^{0}({T_{\mathbb{P}^{g}}}_{|S}(-1)) \cong H^{0}({T_{\mathbb{P}^{g}}}_{|C}(-1))$.\\
By sequence (\ref{h0os-1h0osg+1}), $h^{0}({T_{\mathbb{P}^{g}}}_{|S}(-1))=g+1$ holds. Consider the Euler sequence restricted to $C$ and tensorized by $\mathcal{O}_C(-1)$
$$0 \rightarrow \mathcal{O}_C(-1) \rightarrow H^{0}(\mathcal{O}_S(C))^{*} \otimes \mathcal{O}_C \rightarrow {T_{\mathbb{P}^{10}}}_{|C}(-1) \rightarrow 0.$$
One has that $h^{0}({T_{\mathbb{P}^{10}}}_{|C}(-1))=g+1$ if and only if the cohomology map $\alpha: H^{1}(\mathcal{O}_C(-1)) \rightarrow  H^{0}(\mathcal{O}_S(C))^{*} \otimes H^{1}(\mathcal{O}_C)$ is injective i.e. if and only if its dual
$$H^{0}(\mathcal{O}_S(C)) \otimes H^{0}(\omega_C) \cong \left(H^{0}(\omega_C) \oplus \mathbb{C}\right) \otimes H^{0}(\omega_C) \rightarrow H^{0}({\omega_C}^{\otimes 2})$$
is surjective. It is then sufficient to show that the cup-product map $H^{0}(\omega_C) \otimes H^{0}(\omega_C) \rightarrow H^{0}(\omega_C^{\otimes 2})$ is surjective. Since $(S,L)$ is general, $C$ is non-hyperelliptic, hence the surjectivity holds by Noether's theorem (see e.g. \cite{ACGH}, p. 117) and $h^{1}({T_{\mathbb{P}^{g}}}_{|S}(-2))=0$.\\
% conti a pag. 257
%The inclusion $N_{S/\mathbb{P}^{10}}(-2) \hookrightarrow N_{X/\mathbb{P}^{10}}(-2)$ gives $h^{0}(N_{S/\mathbb{P}^{10}}(-2))=0$. The cohomology sequence associated to the exact sequence
%$$0 \rightarrow T_{S}(-2) \rightarrow {T_{\mathbb{P}^{10}}}_{|S}(-2) \rightarrow N_{S/\mathbb{P}^{10}}(-2) \rightarrow 0$$
%gives $h^{1}(T_S(-2))=0$.
%controllare caso genere 12
\end{proof}
\begin{oss}
\label{h0h1tpgs-2ggeq3}
Note that $h^{0}({T_{\mathbb{P}^{g}}}_{|S}(-2))=h^1({T_{\mathbb{P}^{g}}}_{|S}(-2))=0$ for every $g \geq 3$ (not only for $g \geq 7$).
\end{oss}
\begin{lem}
\label{h2tx-fk3}
Let $f:X \rightarrow \mathbb{P}^1$ be a general K3-type fibration of genus $g$ and let $F$ be a fibre. Then
\begin{equation}
h^2(T_X(-F))=\left\{\begin{array}{cc}
                      0, & g \leq 9 \hbox{ or }g=11 \\
                      1, & g=10 \\
                      3, & g=12 \\
                      2g-22, & g \geq 13
                    \end{array}\right.
\end{equation}
and $h^1(T_X(-2F))=0$ for all $g \geq 7$.
\end{lem}
\begin{proof}
For the first part, one has $\chi(T_S(-C))=\chi(T_S)-\chi({T_S}_{|C})=2g-22$. From Lemma \ref{h1ts-ck3} one gets
$$h^2(T_S(-C))=\left\{\begin{array}{cc}
                      0, & g \leq 9 \hbox{ or } g=11\\
                      1, & g=10 \\
                      3, & g=12 \\
                      2g-22, & g \geq 13.
                    \end{array}\right.
$$
Notation as above, let $E_i \doteq \epsilon^{-1}(p_i)$ be the exceptional divisors of the blow-up $\epsilon:X \rightarrow S$ and $E \doteq \sum_i E_i$. From \cite{S}, Example 3.4.13 (iv) one has $N_{\epsilon} \cong \mathcal{O}_E(-E)$, thus there is an exact sequence
$$0 \rightarrow T_X(-F) \rightarrow \epsilon^{*}T_S(-F) \rightarrow \mathcal{O}_{E}(-E-F) \rightarrow 0$$
whose cohomology sequence gives $h^2(T_X(-F))=h^2(\epsilon^{*}T_S(-F))$.\\
By the definition of $F$ one has $\epsilon^{*}T_S(-F) \cong \epsilon^{*}(T_S(-C)) \otimes \mathcal{O}_X(E)$. From \cite{H}, V Proposition 3.4, it follows that $\epsilon_{*}\mathcal{O}_X(E) \cong \mathcal{O}_S$ and $R^1 \epsilon_{*}\mathcal{O}_X(E) =0$, hence the low-degree terms exact sequence associated to the Leray spectral sequence for the sheaf $\epsilon^{*}(T_S(-C)) \otimes \mathcal{O}_X(E)$ gives $h^2\left(\epsilon^{*}(T_S(-C)) \otimes \mathcal{O}_X(E)\right)=h^{2}(T_S(-C)).$\\
%conti a pag. 221
Let $C_1,C_2$ be general curves in $|L|$. Under the assumptions of Theorem \ref{equivalencexf1xf2sc1c2} one has $\text{im}(\alpha) \cong \text{im}(\beta)$ in the commutative diagram
\begin{equation}
\label{defxf1f2defsc1c2}
\xymatrix{0 \ar[r] & \ker \alpha \ar[r] & \text{Def}^{\prime}_{(X,F_1,F_2)}(k[\epsilon]) \ar[r]^{\alpha} \ar[d]^{\cong} & \text{Def}^{\prime}_{F_1}(k[\epsilon]) \oplus \text{Def}^{\prime}_{F_2}(k[\epsilon]) \ar[d]^{\cong} \\ 0 \ar[r] & \ker \beta \ar[r] & \text{Def}^{\prime}_{(S,C_1,C_2)}(k[\epsilon]) \ar[r]^{\beta} & \text{Def}^{\prime}_{C_1}(k[\epsilon]) \oplus \text{Def}^{\prime}_{C_2}(k[\epsilon])}
\end{equation}
thus $\ker \alpha \cong \ker \beta$.\\
Since $F_1 \cdot F_2=0$, one has $\text{Def}^{\prime}_{(X,F_1,F_2)}(k[\epsilon]) \cong H^{1}(T_X \langle F_1 \cup F_2 \rangle)$, $H^{1}(T_{F_1 \cup F_2 }) \cong H^{1}(T_{F_1}) \oplus H^{1}(T_{F_2})$ and the first row of diagram (\ref{defxf1f2defsc1c2}) coincides with the cohomology sequence of (\ref{ixyty<x>}) at the $H^{1}$-level
$$0 \rightarrow H^{1}(T_X(-2F)) \rightarrow H^{1}(T_X \langle F_1 \cup F_2 \rangle) \rightarrow H^{1}(T_{F_1 \cup F_2}).$$
In particular $\ker \alpha \cong H^{1}(T_X(-2F))$. Now consider the cohomology sequence
$$0 \rightarrow H^{1}(T_S(-2C)) \rightarrow H^{1}(T_S \langle C_1 \cup C_2 \rangle) \xrightarrow{\gamma} H^{1}(T_{C_1 \cup C_2}).$$
Since $\text{Def}^{\prime}_{(S,C_1,C_2)}(k[\epsilon]) \subset H^{1}(T_S \langle C_1 \cup C_2 \rangle)$ and $\beta = \gamma_{|\text{Def}^{\prime}_{(S,C_1,C_2)}(k[\epsilon])}$, one has $\ker \beta \subseteq H^{1}(T_S(-2C))=0$, where the last equality follows from Lemma \ref{h1ts-ck3}.
\end{proof}
In order to apply Proposition \ref{dimlocusvaigeq0}, we have to prove that, for a general K3-type fibration $f$ of genus $g$, the image of $\Psi_f$ is contained in the smooth locus on $\overline{M}_g$. The argument in the proof involves hyperelliptic K3 surfaces.
\begin{teo}[\cite{SA}]
Let $S$ be a smooth K3 surface and $L \in \text{Pic}(S)$ a globally generated line bundle such that $L^2 > 0$. Let $\phi \doteq \phi_{|L|}:S \rightarrow \mathbb{P}^g$ be the associated morphism, where $L^2=2g-2$. Then there are only two possibilities:
\begin{itemize}
\item[(i)] if $|L|$ contains a nonhyperelliptic curve, then $\phi$ is birational onto a surface of degree $2g-2$;
\item[(ii)] if $|L|$ contains a hyperelliptic curve, then $\phi$ is a generically 2-to-1 mapping of $S$ onto a surface $W$ of degree $g - 1$. The surface $W$ is one among $\mathbb{P}^2$ and the Hirzebruch surfaces $\mathbb{F}_n$ with $n=0,1,2,3,4$. The branching curve belongs to the linear system $|-2K_W|$.
\end{itemize}
\end{teo}
In the second case, $|L|$ will be called a \emph{hyperelliptic linear system} and $S$ a \emph{hyperelliptic K3 surface} (\emph{of genus $g$}).\\
The surfaces $\mathbb{F}_n$ are the projective bundles $\mathbb{P}_{\mathbb{P}^1}\left(\mathcal{O}_{\mathbb{P}^1}(n) \oplus \mathcal{O}_{\mathbb{P}^1}\right)$. Let $H$ be the class of the tautological line bundle $\mathcal{O}_{\mathbb{F}_n}(1)$ in $\text{Pic}(\mathbb{F}_n)$, let $F$ be a fibre of the projection over $\mathbb{P}^1$ and, for $n \geq 1$, let $B$ be the unique irreducible curve on $\mathbb{F}_n$ with negative self-intersection. One has $H^2=n$, $F^2=0$, $H \cdot F=1$ and $B^2=-n$. The canonical classes are $K_{\mathbb{F}_0}=-2F-2H$ and $K_{\mathbb{F}_n}=-(n+2)F-2B$, $n \geq 1$.\\
If $S$ is a double cover of $F=\mathbb{F}_n$, the hyperelliptic linear system is
$$|L_n| \doteq \left\{\begin{array}{cc}
                      \left|\phi^{*}\left(\left(\frac{n}{2}+\frac{g-1}{2}\right)F+B\right)\right|, & n \neq 0, \\
                      \left|\phi^{*}\left(\left(\frac{g-1}{2}\right)F+H\right)\right|, & n=0
                    \end{array}\right.$$
and the $g^1_2$ is cut out by the elliptic pencil $|E|$ where $E \doteq \phi^{*}(F)$.\\
The pairs $(S,L_n)$, with $S \rightarrow \mathbb{F}_n$ as above, define a substack $\mathcal{A}_{g,n} \subset \mathcal{F}_g$ and the union of these is just the \emph{hyperelliptic substack} of $\mathcal{F}_g$ i.e. the substack parameterizing pairs $(S,L)$ such that $S$ is a hyperelliptic surface and $|L|$ is the hyperelliptic linear system.\\
It turns out that in a certain sense the double covers of $\mathbb{F}_2$ and $\mathbb{F}_3$ are degenerate cases of double covers of $\mathbb{F}_0$ and $\mathbb{F}_1$ (see \cite{RE}, Theorem 3.5 for a precise statement).\\
As a consequence, if $g$ is even, the hyperelliptic substack consists of the single component $\mathcal{A}_{g,1}$, whereas if $g$ is odd it has two components $\mathcal{A}_{g,0}$ and $\mathcal{A}_{g,4}$. These components are in all cases of codimension 1 in $\mathcal{F}_g$ i.e. they are 18-dimensional.\\
Define $\mathcal{D}_{g,n} \subset \mathcal{P}_g$ as the inverse image of $\mathcal{A}_{g,n}$ by the projection $\pi_g$. $\mathcal{D}_{g,1}$, $\mathcal{D}_{g,0}$ and $\mathcal{D}_{g,4}$ are (18+g)-dimensional substacks of $\mathcal{P}_g$.
\begin{prop}
\label{generalfk3fibresnoautomorphisms}
Let $f$ be a general K3-type fibration genus $g$. Then the image of $\Psi_f$ is contained in the smooth locus of $\overline{M}_g$.
\end{prop}
\begin{proof}
For low genera the statement can be proved with a straightforward adaptation of the proof of \cite{CU}, Lemma 3.4, which refers to $g=10$. If the morphism $c_g$ is dominant, which happens for $3 \leq g \leq 9$ and $g=11$, then the statement follows from the fact that there are no irreducible components of the singular locus of $\overline{M}_g$ of codimension 1 (see \cite{HMO}, pp. 53-54). Hence we can assume the result for $g \leq 11$.\\
%Even if for our purposes it is sufficient to state that the argument works for $g \leq 11$, it is worthwhile to say that it works till the dimension of $\mathcal{K}_g$ is sufficiently large with respect to the maximal dimension $s_g$ of the irreducible components of the singular locus of ${M}_g$, which are stated in \cite{CO}.\\
Let $S$ be a general hyperelliptic K3 surface of genus $g$ and let $|L|$ be the hyperelliptic linear system. It is sufficient to prove that there is a linear pencil $\mathcal{P} \subset |L|$ which consists entirely of curves having no automorphisms apart from the hyperelliptic involution. Indeed, if this is the case, the polarized surface $(S,L)$ can be deformed to a polarized non-hyperelliptic surface $(S',L')$ having a pencil $\mathcal{P'} \subset |L'|$ consisting entirely of curves without automorphisms.\\
Let $\mathcal{H}_g \subset \mathcal{M}_g$ be the hyperelliptic locus and let $\mathcal{L}_g \subset \mathcal{H}_g$ be the locus consisting of hyperelliptic curves with extra automorphisms.
Let $k$ be the dimension of the locus in $|L|$ parameterizing smooth hyperelliptic curves with extra automorphisms. We want to show that $k \leq g-2$, so that the general pencil $\mathcal{P} \subset |L|$ does not contain smooth curves with extra automorphisms.\\
Let $g$ be odd, let $\mathcal{D}_{g,0}$ as in the definition above, and let $\mathcal{D}'_{g,0} \subset \mathcal{D}_{g,0}$ be the sublocus parameterizing pairs $(S,C)$ such that $C$ is smooth and has extra automorphisms. Since $\mathcal{L}_g$ has dimension $g$ (see \cite{GS}), each fibre $c_g^{-1}([C])$ of the map $c_g:\mathcal{D}'_{g,0} \rightarrow \overline{\mathcal{M}}_g$ has dimension at least $18+k-g$.\\
Let $(S,L)$ be a general hyperelliptic K3 surface in $\mathcal{A}_{g,0}$, let $\phi \doteq \phi_{|L|}:S \rightarrow \mathbb{F}_0 = \mathbb{P}^1 \times \mathbb{P}^1$ be the double cover ramified over a curve $G \in |4F+4H|$ and let $C \in |L|$. We want to show that $c_g^{-1}([C])$ has dimension at most $16$, so that $18+k-g \leq 16$ i.e. $k \leq g-2$.\\
One has $\dim c_g^{-1}([C]) \leq h^{1}(T_S(-C))$ since the tangent space to $c_g^{-1}([C])$ at the point $(S,C)$ is $H^{1}(T_S(-C))$.
%anche se la fibra è singolare questo è vero
Let $E \doteq \phi^{*}F$. The elliptic pencil $|E|$ defines a rational fibration $\varphi \doteq \varphi_{|E|}:S \rightarrow \mathbb{P}^1$. Consider the short exact sequence
$$0 \rightarrow \varphi^{*}\omega_{\mathbb{P}^1}(C) \rightarrow \Omega^{1}_S(C) \rightarrow \Omega^1_{S/\mathbb{P}^1}(C) \rightarrow 0.$$
One has $\varphi^{*}\omega_{\mathbb{P}^1}(C) \cong \mathcal{O}_S(C-2E)$ and $C-2E$ is linearly equivalent to $\frac{g-5}{2}E+\phi^{*}H$, which is big and nef (for $g \geq 7$), hence by the Kawamata-Viehweg vanishing theorem one has the exact sequence
\begin{equation}
\label{h0varphiomegasp1c}
0 \rightarrow H^{0}(\varphi^{*}\omega_{\mathbb{P}^1}(C)) \rightarrow H^{0}(\Omega^{1}_S(C)) \rightarrow H^{0}(\Omega^1_{S/\mathbb{P}^1}(C)) \rightarrow 0
\end{equation}
and $h^{0}(\varphi^{*}\omega_{\mathbb{P}^1}(C))=\chi(\mathcal{O}_S(C-2E))=g-3$.\\
Since $|F|$ has no singular curves, the singular curves of $|E|$ are pull-backs by $\phi$ of curves in $|F|$ which are tangents to the branch divisor $G$. Since $G$ is general in its linear system, the tangencies are simple, thus the singularities of the fibration $\varphi$ are nodal. We then have a short exact sequence
\begin{equation}
\label{omega1sp1comegasp1c}
0 \rightarrow \Omega^1_{S/\mathbb{P}^1}(C) \rightarrow \omega_{S/\mathbb{P}^1}(C) \rightarrow N \rightarrow 0
\end{equation}
where $N$ is a torsion sheaf supported on the set of $24=12\chi(\mathcal{O}_S)-K^2_S$ singular points $q_1,...,q_{24}$ of the singular fibres of $\varphi$ and having stalk $\mathbb{C}$ at each of them. Sequence (\ref{omega1sp1comegasp1c}) gives
\begin{equation}
\label{omega1sp1ccongiq1q24}
\Omega^1_{S/\mathbb{P}^1}(C) \cong \mathcal{I}_{q_1,...,q_{24}} \otimes \omega_{S/\mathbb{P}^1}(C).
\end{equation}
By equality (\ref{relativecanonicalsheaf}) one has $\omega_{S/\mathbb{P}^1}(C) \cong \omega_S(C) \otimes \varphi^{*}\omega^{-1}_{\mathbb{P}^1} \cong \mathcal{O}_S(C+2E)$. One can again apply Kawamata-Viehweg on $\mathcal{O}_S(C+2E)$ and obtain $h^{0}(\mathcal{O}_S(C+2E))=\chi(\mathcal{O}_S(C+2E))=g+5$.\\
%Consider the following reducible curve of the form $D \doteq F_1 \cup F_2 \cup \dots \cup F_{\frac{g-1}{2}} \cup H_1$ in the linear system $\left|\left(\frac{g-1}{2}F+H\right)\right|$ on $\mathbb{F}_0=\mathbb{P}^1 \times \mathbb{P}^1$:\\
%\begin{picture}(200,120)(1,1)
%\label{curvariducibile}
%\put(1,80){\line(1,0){100}}
%\put(110,80){\textsf{$F_{\frac{g-1}{2}}$}}
%\put(1,65){\line(1,0){100}}
%\put(110,65){\textsf{$F_{\frac{g-3}{2}}$}}
%\put(50,50){\textsf{\dots}}
%\put(100,50){\textsf{$\frac{g-1}{2}F$}}
%\put(1,35){\line(1,0){100}}
%\put(110,35){\textsf{$F_{2}$}}
%\put(1,20){\line(1,0){100}}
%\put(110,20){\textsf{$F_{1}$}}
%\put(80,05){\line(0,1){90}}
%\put(75,105){\textsf{$H_1$}}
%\end{picture}\\
Let $i=\min \left\{\frac{g+3}{2}+1,24\right\}$. Since $G$ is general, there is a reducible curve $A \in \left|\frac{g+3}{2}F+H\right| \subset \mathbb{P}^1 \times \mathbb{P}^1$, of the form $F_1+F_2+ \dots +F_{\frac{g+3}{2}}+H$, which separates the points $\phi(q_1),...,\phi(q_i)$, thus the curve $\phi^{*}A \in |C+2E|$ separates the points $q_1,...,q_i$ and so $q_1,...,q_{24}$ impose at least $i$ conditions to the linear system $|C+2E|$. In particular for $g \geq 11$ they impose at least $8$ conditions, hence isomorphism (\ref{omega1sp1ccongiq1q24}) gives $h^{0}(\Omega^1_{S/\mathbb{P}^1}(C)) \leq h^{0}(\omega_{S/\mathbb{P}^1}(C))-8=g-3$. Sequence (\ref{h0varphiomegasp1c}) then gives $h^{0}(\Omega^1_S(C)) = h^{0}(\varphi^{*}\omega_{\mathbb{P}^1}(C))+h^{0}(\Omega^1_{S/\mathbb{P}^1}(C)) \leq 2g-6$. Serre duality gives $h^{0}(\Omega^1_S(C))=h^{2}(T_S(-C))$. Since $\chi(T_S(-C))=h^{2}(T_S(-C))-h^{1}(T_S(-C))=2g-22$, one gets $h^{1}(T_S(-C)) \leq 16$ as we wanted.\\
A similar argument yields the same property if $g$ is even. In this case the role of $A$ is played by a reducible curve of the form $F_1+F_2+ \dots + F_{\frac{g}{2}+2}+B \in \left|\left(\frac{g}{2}+2\right)F+B\right|$ on $\mathbb{F}_1$. In this way we have proved that a general pencil $\mathcal{P}$ on a general hyperelliptic K3 surface does not contain smooth curves having extra automorphisms apart from the hyperelliptic involution. We have to prove that the same is true for the singular curves in $\mathcal{P}$.\\
Note that singular curves in the linear system $\left|\frac{g-1}{2}F+H\right|$ (if $g$ is odd) or $\left|\frac{g}{2}F+B\right|$ (is $g$ is even) are all reducible,
%perché sono curve razionali
hence singular irreducible curves in $|L|$ are pull-backs of curves which are tangent to the branch divisor $G$. Since $G$ is general, the general tangent curve to $G$ has only one (simple) tangency point, thus singular irreducible curves in $\mathcal{P}$ must have one ordinary node and no other singularities.
%il genere aritmetico è 0 e il genere della desingolarizzazione decresce
An automorphism of such a curve can be seen as an automorphism of its desingularization, which is a smooth curve of genus $g-1$, fixing (as a set) the inverse images of the singular point. Hence the locus of irreducible $1$-nodal hyperelliptic curves with extra automorphisms has dimension less than or equal to $\dim (\mathcal{L}_{g-1})=g-1$.
However, $\mathcal{P}$ could contain also hyperelliptic reducible curves which are pull-backs of connected reducible curves. Fix a (general) hyperelliptic reducible curve $C$ of this kind. One has $C = C_1 \cup C_2 \cup ... \cup C_l$,
%C_i non può essere razionale perché ogni componente della curva riducibile sotto interseca G in almeno 4 punti
where $C_i$ are irreducible hyperelliptic curves. Since $C$ is general, one has $C_i \neq C_j$ for $i \neq j$, thus an automorphism of $C$ is a $l$-tuple $(\gamma_1,...,\gamma_l)$, where $\gamma_i$ is an automorphism of $C_i$ fixing (as a set) the intersection points of $C_i$ with each other component $C_j$. Since $C$ is the pull-back of a connected curve, each component $C_i$ has to intersect at least one other component in at least two points. Since $p_a(C)=g$ one has that $\sum_i g(C_i) \leq g-1$. Using this fact, one finds that the locus of reducible hyperelliptic curves with extra automorphisms has dimension less than or equal to $g-1$.
%cioè un automorfismo non scambia componenti perchè non ci sono componenti tra loro isomorfe
%il punto di intersezione delle due componenti sotto non è di ramificazione per la C generica
In conclusion, if $k$ is the dimension of the locus in $|L|$ parameterizing singular hyperelliptic curves with extra-automorphisms, applying the same argument of the smooth case, one finds again $k \leq g-2$, hence singular curves in $\mathcal{P}$ have no extra automorphisms.
%per quanto detto sopra, la generica singolare è 1-nodata
This completes the proof.
%We claim that the family of primitively polarized K3 surfaces of genus $g$ containing a curve isomorphic to $\phi^{*}(D)$ has dimension $6$. Suppose that a branch curve $G'$ in the linear system $|4F+4B|$ cuts the points $p_1,p_2,...,p_{\frac{g-1}{2}}$ and $q$ on $D$. Then in particular $G'$ meets $l$ in $\frac{g-1}{2} \geq 6$ points (since $g \geq 12$), thus, since $G' \cdot l=4$, $G'$ must contain $l$ as a component. This implies that $G'$ is a reducible branching curve like the one drawn in Figure \ref{branchingcurve1}.
\end{proof}
Let $f$ be a K3-type fibration of genus $g$. It is easy to see that the number of negative summands in the splitting of $Ext^1_f(\Omega^1_{X/\mathbb{P}^1},\mathcal{O}_X)$ is greater than or equal to the codimension of the locus $\mathcal{K}_g$ in $\mathcal{M}_g$.\\
For $g \leq 9$ and $g=11$ the statement is obvious as the map $c_g$ is dominant.\\
For the remaining cases consider the exact sequence (\ref{sequenzanormale2tensorizzata}) with $k=-1$:
\begin{equation}
\label{sequenzanormale2tensorizzatacaso-1}
0 \rightarrow N_{\Psi_f}(-1) \rightarrow R^1 f_{*}T_X(-F) \rightarrow R^1 f_{*}\mathcal{O}_X \otimes \mathcal{O}_{\mathbb{P}^1}(1) \rightarrow 0.
\end{equation}
By (\ref{tp1tensorr1fstarox}) one has $h^1(R^1 f_{*}\mathcal{O}_X \otimes \mathcal{O}_{\mathbb{P}^1}(1))=0$, from which
\begin{equation}
h^1(N_{\Psi_f}(-1)) \geq h^1(R^1 f_{*}T_X(-F))=h^2(T_X(-F)).
\end{equation}
From Lemma \ref{h2tx-fk3} one has
\begin{equation}
\label{h1npsifgeq}
h^1(N_{\Psi_f}(-1)) \geq \left\{\begin{array}{cc}
                      1, & g=10 \\
                      3, & g=12 \\
                      2g-22, & g \geq 13.
                    \end{array}\right.
\end{equation}
On the other hand, by Proposition \ref{h0h1ext1f} all the negative $a_i$ equal -1, hence by (\ref{npsif-1ext1f-1}) the number $h^1(N_{\Psi_f}(-1))$ is exactly the number of negative $a_i$, while the values in the right term are exactly the codimensions of $\mathcal{K}_g$ in $\mathcal{M}_g$, as given by (\ref{codimkg}).\\
Although deformations of $f$ are unobstructed, the number of negative summands in the splitting of $Ext^1_f(\Omega^1_{X/\mathbb{P}^1},\mathcal{O}_X)$ could be strictly greater than the codimension of the locus $\mathcal{K}_g$ in $\mathcal{M}_g$. The following proposition proves that this is not the case if $f$ is a \emph{general} K3-type fibration:
\begin{prop}
\label{equalsthecodimensionofkg}
Let $f$ be a general K3-type fibration of genus $g$. Then the number of negative summands in the splitting of $Ext^1_f(\Omega^1_{X/\mathbb{P}^1},\mathcal{O}_X)$ equals the codimension of the locus $\mathcal{K}_g$ in $\mathcal{M}_g$.
\end{prop}
\begin{proof}
As has just been pointed out, the codimension of $\mathcal{K}_g$ in $\mathcal{M}_g$ is $0$ for $g \leq 9$ or $g=11$. By Proposition \ref{generalk3typefibrationfreeiff} $f$ is free for these genera.\\
For the remaining cases, by Proposition \ref{generalfk3fibresnoautomorphisms} one has that $\overline{M}_g$ is smooth along the image of $\Psi_f$ for $f$ general, thus $\Psi_f^{*}T_{\overline{M}_g}$ is locally free and isomorphic to $Ext^1_f(\Omega^1_{X/\mathbb{P}^1},\mathcal{O}_X)$.\\
%This means that $H^{0}(Ext^1_f(\Omega^1_{X/\mathbb{P}^1},\mathcal{O}_X))$ is exactly the tangent space at $[\Psi_f]$ to the irreducible component $V \subset \text{Mor}(\mathbb{P}^1,\overline{M}_g)$ containing the point $[\Psi_f]$ (see Proposition \ref{tmoryxf}).\\
%Questo serviva per dimostrare che [\Psi_f] corrisponde a un punto generico di V, ma questo è vero
Let $V \subset \text{Mor}(\mathbb{P}^1,\overline{M}_g)$ be the irreducible component containing the point $[\Psi_f]$.
By Proposition \ref{h0h1ext1f} one has $h^{1}(\Psi^{*}_f T_{\overline{M}_g})=0$, thus Proposition \ref{dimlocusvaigeq0} can be applied. Since $\dim \hbox{ Locus}(V)=\dim \mathcal{K}_g$
% probabilmente è vero che sono uguali, ma a noi interessa che si intersechino in un aperto non vuoto
the statement follows.
\end{proof}
We now have all the information to explicitly write down the splitting of the locally free sheaves $Ext^1_f(\Omega^1_{X/\mathbb{P}^1},\mathcal{O}_X)$ and $N_{\Psi_f} \doteq \bigoplus_{i=1}^{3g-4}\mathcal{O}_{\mathbb{P}^1}(b_i)$ for a general K3-type fibration $f$.
\begin{teo}
\label{decomposizioneext1fk3}
Let $f$ be a general K3-type fibration of genus $g \geq 7$. Then
$$Ext^1_f(\Omega^1_{X/\mathbb{P}^1},\mathcal{O}_X) \cong \left\{\begin{array}{lc}
\mathcal{O}_{\mathbb{P}^1}(2) \oplus \mathcal{O}_{\mathbb{P}^1}(1)^{\oplus 21-g} \oplus \mathcal{O}_{\mathbb{P}^1}^{\oplus 4g-25}, & 7 \leq g \leq 9\\
\mathcal{O}_{\mathbb{P}^1}(2) \oplus \mathcal{O}_{\mathbb{P}^1}(1)^{\oplus 12} \oplus \mathcal{O}_{\mathbb{P}^1}^{\oplus 13} \oplus \mathcal{O}_{\mathbb{P}^1}(-1), & g=10\\
\mathcal{O}_{\mathbb{P}^1}(2) \oplus \mathcal{O}_{\mathbb{P}^1}(1)^{\oplus 10} \oplus \mathcal{O}_{\mathbb{P}^1}^{\oplus 19}, & g=11 \\
\mathcal{O}_{\mathbb{P}^1}(2) \oplus \mathcal{O}_{\mathbb{P}^1}(1)^{\oplus 12} \oplus \mathcal{O}_{\mathbb{P}^1}^{\oplus 17} \oplus \mathcal{O}_{\mathbb{P}^1}(-1)^{\oplus 3}, & g=12\\
\mathcal{O}_{\mathbb{P}^1}(2) \oplus \mathcal{O}_{\mathbb{P}^1}(1)^{\oplus g-1} \oplus \mathcal{O}_{\mathbb{P}^1}^{\oplus 19} \oplus \mathcal{O}_{\mathbb{P}^1}(-1)^{\oplus 2g-22}, & g \geq 13.
             \end{array}\right.$$
\end{teo}
\begin{proof}
Proposition \ref{h0h1ext1f} gives $h^{0}(Ext^1_f(\Omega^1_{X/\mathbb{P}^1},\mathcal{O}_X))=2g+20$ and moreover tells that all the negative $a_i$ equal $-1$. Proposition \ref{equalsthecodimensionofkg} and Theorem \ref{cg} give the number of negative summands. Proposition \ref{unsoloai=2} applied to Lemma \ref{h2tx-fk3} gives that there is only one summand $a_i \geq 2$ and it equals $2$.\\
These facts uniquely determine the splitting type of $Ext^1_f(\Omega^1_{X/\mathbb{P}^1},\mathcal{O}_X)$.
\end{proof}
%By Proposition \ref{npsiflocallyfreek3} the sheaf $N_{\Psi_f}$ is locally free. Its splitting type $\bigoplus_{i=1}^{3g-4}\mathcal{O}_{\mathbb{P}^1}(b_i)$ is determined by the following
%\begin{lem}
%\label{morfismofascisuP1}
%Let $\alpha: \bigoplus_{i=1}^k \mathcal{O}_{\mathbb{P}^1}(a_i) \rightarrow \bigoplus_{j=1}^{k'} \mathcal{O}_{\mathbb{P}^1}(c_j)$ be a morphism of locally free sheaves over $\mathbb{P}^1$. If $c_h < \min_{i}\left\{a_i\right\}$ for some $h$, then $\emph{im}(\alpha)=0$ on $\mathcal{O}_{\mathbb{P}^1}(c_h)$.
%\end{lem}
%\begin{proof}
%The morphism $\alpha$ induces componentwise morphisms $\alpha_j$ over each summand $\mathcal{O}_{\mathbb{P}^1}(c_j)$.
%If $c_h < \min_{i}\left\{a_i\right\}$ then $Hom(\bigoplus_{i=1}^k \mathcal{O}_{\mathbb{P}^1}(a_i), \mathcal{O}_{\mathbb{P}^1}(c_h)) \cong Hom(\bigoplus_{i=1}^k \mathcal{O}_{\mathbb{P}^1}(a_i) \otimes \mathcal{O}_{\mathbb{P}^1}(-c_h),\mathcal{O}_{\mathbb{P}^1})$ has not global sections and thus $\text{im}(\alpha_h)=0$, that is $\text{im}(\alpha)=0$ on $\mathcal{O}_{\mathbb{P}^1}(c_h)$.
%\end{proof}
\begin{cor}
Let $f$ be a general K3-type fibration of genus $g \geq 7$. Then sequence \emph{(\ref{tbkfext1fnpsif})} splits i.e.  $Ext^1_f(\Omega^1_{X/\mathbb{P}^1},\mathcal{O}_X) \cong \mathcal{O}_{\mathbb{P}^1}(2) \oplus N_{\Psi_f}$.
\end{cor}
\begin{proof}
Proposition \ref{h0h1ext1f} and equations (\ref{relationcohomologyext1npsif}) and (\ref{npsif-1ext1f-1}) give $h^{0}(N_{\Psi_f})=2g+17$, $h^{1}(N_{\Psi_f})=0$ and $h^{1}(N_{\Psi_f}(-1))=h^{1}(Ext^1_f(\Omega^1_{X/\mathbb{P}^1},\mathcal{O}_X)(-1))$. Moreover, by Proposition \ref{unsoloai=2} one has $b_i \leq 1$ for all $i$. These facts and Theorem \ref{decomposizioneext1fk3} uniquely determine the splitting of $N_{\Psi_f}$. The statement follows.
\end{proof}
\end{section}
\begin{section}{Geometric interpretation of Theorem \ref{decomposizioneext1fk3}}
\begin{lem}
\label{inclusionmg}
Let $f:X \rightarrow \mathbb{P}^1$ be a non-isotrivial fibration of genus $g$ defined by a linear pencil $\Lambda \subset U \subset |C|$ on a smooth projective surface $S$ such that the curves in $\Lambda$ are smooth at the base points, where $U$ is the open set where the moduli map $\mu:U \rightarrow \overline{M}_g$ is defined. If $\mu$ is generically finite, then there is an inclusion $\mathcal{O}_{\mathbb{P}^1}(2) \oplus \mathcal{O}_{\mathbb{P}^1}(1)^{\oplus \dim|C|-1} \hookrightarrow \Psi_f^{*}T_{{{\overline{M}}}_{g}}$.
\end{lem}
\begin{proof}
Let $l$ be the line corresponding to $\Lambda$ in the projective space $\mathbb{P}^{\dim|C|} \cong |C|$. Since $\mu$ is generically finite and $\mu_{|l}=\Psi_f$ (by definition and the fact that the curves in $\Lambda$ are smooth at the base points)
%piccolo abuso di notazione, \Lambda non è necessariamente contenuto in U
one has
$$\mathcal{O}_{\mathbb{P}^1}(2) \oplus \mathcal{O}_{\mathbb{P}^1}(1)^{\oplus \dim|C|-1} \cong {T_{\mathbb{P}^{\dim|C|}}}_{|l} \cong {T_U}_{|l} \hookrightarrow {\mu^{*}{{T_{\overline{M}}}_{g}}}_{|l} \cong \Psi_f^{*}T_{{{\overline{M}}}_{g}}.$$
\end{proof}
In the case of K3-type fibrations, one has that $\mu$ is finite by \cite{CK}, Proposition 1.2, thus Lemma \ref{inclusionmg} applies.
The inclusion in Lemma \ref{inclusionmg} induces an inclusion of the spaces of global sections $$H^{0}(\mathcal{O}_{\mathbb{P}^1}(2) \oplus \mathcal{O}_{\mathbb{P}^1}(1)^{\oplus g-1}) \hookrightarrow H^{0}(Ext^1_f(\Omega^1_{X/\mathbb{P}^1},\mathcal{O}_X)) \cong \text{Def}^{\prime}_f(k[\epsilon])$$
which admits a nice geometrical interpretation. Indeed $H^{0}\left(\mathcal{O}_{\mathbb{P}^1}(2) \oplus \mathcal{O}_{\mathbb{P}^1}(1)^{\oplus g-1}\right)$ can be seen as the vector subspace of $\text{Def}'_f(k[\epsilon])$ parameterizing first order deformations of the morphism $f$ obtained by keeping $S$ fixed and varying the linear pencil $\Lambda$ inside $|C|$.\\
The inclusion in the lemma is as a direct summand for $g \geq 7$ (see Theorem \ref{decomposizioneext1fk3}). In other words, the first $g$ summands in the splitting of $Ext^1_f(\Omega^1_{X/\mathbb{P}^1},\mathcal{O}_X)$ for $g \geq 7$ can be interpreted as those "coming from" deformations of $f$ inducing a trivial deformation of the surface $S$.\\
In the cases $g=7,8,9,10,12$ there are more summands of the type $\mathcal{O}_{\mathbb{P}^1}(1)$ in the splitting of $Ext^1_f(\Omega^1_{X/\mathbb{P}^1},\mathcal{O}_X)$ and their number equal the dimension of the general fibre of the morphism $c_g$.\\
In conclusion, also keeping in mind Proposition \ref{equalsthecodimensionofkg}, the splitting of $Ext^1_f(\Omega^1_{X/\mathbb{P}^1},\mathcal{O}_X)$ for $g \geq 7$ can be restated as follows:
\begin{teo}
\label{splittingrestatedk3}
Let $f$ be a general K3-type fibration of genus $g \geq 7$, let $a_g$ be the dimension of the general fibre of the morphism $c_g$, and $b_g$ the codimension of $\mathcal{K}_g$ in $\mathcal{M}_g$. Then
$$Ext^1_f(\Omega^1_{X/\mathbb{P}^1},\mathcal{O}_X) \cong \mathcal{O}_{\mathbb{P}^1}(2) \oplus \mathcal{O}_{\mathbb{P}^1}(1)^{\oplus g-1} \oplus \mathcal{O}_{\mathbb{P}^1}(1)^{\oplus a_g} \oplus \mathcal{O}_{\mathbb{P}^1}^{\oplus 2g-3-a_g-b_g} \oplus \mathcal{O}_{\mathbb{P}^1}(-1)^{\oplus b_g}.$$
\end{teo}
\noindent
One can consider the vector space $H^{0}(Ext^1_f(\Omega^1_{X/\mathbb{P}^1},\mathcal{O}_X)(-1))$, which is the subspace of $H^{0}(Ext^1_f(\Omega^1_{X/\mathbb{P}^1},\mathcal{O}_X))$ parameterizing first order deformations of $f$ inducing trivial deformations on a fixed fibre $F=f^{*}\mathcal{O}_{\mathbb{P}^1}(1)$.\\
From Theorem \ref{splittingrestatedk3} one has $h^{0}(Ext^1_f(\Omega^1_{X/\mathbb{P}^1},\mathcal{O}_X)(-1))=2+(g-1)+a_g$. The first summand is the number of parameters counting automorphisms of $\mathbb{P}^1$ keeping a point fixed, the second is the number of parameters counting lines in $|C| \cong \mathbb{P}^g$ passing through a fixed point, the third is the number of parameters counting K3 surfaces containing a curve isomorphic to $F$.
\begin{oss}
Theorem \ref{splittingrestatedk3} shows that, for $g=7,8,9,11$, the morphism $\Psi_f$ is a so-called \emph{minimal free morphism}. Given a projective variety $X$ of dimension $n$ and a free morphism $h:\mathbb{P}^1 \rightarrow X$, $h$ is said to be \emph{minimal} if the splitting of $h^{*}T_X$ is:
$$h^{*}T_X \cong \mathcal{O}_{\mathbb{P}^1}(2) \oplus \mathcal{O}_{\mathbb{P}^1}(1)^{\oplus r} \oplus \mathcal{O}_{\mathbb{P}^1}^{\oplus n-r-1}$$
for some nonnegative integer $r$. The number $r+1$ equals the dimension of the subvariety of $X$ swept out by deformations of $h$ passing through a fixed general point of $X$ (see \cite{DE}, p. 93).\\
\end{oss}
Note that for $g \leq 6$ the rank of $Ext^1_f(\Omega^1_{X/\mathbb{P}^1},\mathcal{O}_X)$, which is $3g-3$, is too small with respect to the dimension of the space of global sections, which is $2g+20$, hence more summands of the kind $\mathcal{O}_{\mathbb{P}^1}(a_i)$, $a_i \geq 2$, must appear in the splitting of $Ext^1_f(\Omega^1_{X/\mathbb{P}^1},\mathcal{O}_X)$ and the structure shown in Theorem \ref{splittingrestatedk3} is lost. Moreover, $\Psi_f$ is not a minimal morphism anymore.\\ \\
%This can be used to give a slight improvement of Lemma \ref{h0ncpgk3}.
%\begin{prop}
%Let $C$ be a general hyperplane section of a general primitively polarized K3 surface of genus $g=6$. Then %$h^{0}(N_{C/\mathbb{P}^{g-1}}(-2))=1$.
%\end{prop}
%\begin{proof}
%
%\end{proof}
%Nevertheless we are able to write down this splitting also for $g=6$.
%\begin{prop}
%Let $f$ be a general K3-type fibration of genus $6$. Then
%$$Ext^1_f(\Omega^1_{X/\mathbb{P}^1},\mathcal{O}_X) \cong \mathcal{O}_{\mathbb{P}^1}(2) \oplus \mathcal{O}_{\mathbb{P}^1}(1)^{\oplus 21-g} \oplus \mathcal{O}_{\mathbb{P}^1}^{\oplus 4g-25}.$$
%\end{prop}
%\begin{proof}
%\end{proof}
%The proof provides also a slight improvement of Lemma \ref{h0ncpgk3}:
%\begin{cor}
%Let $C$ be a general hyperplane section of a general primitively polari-zed K3 surface of genus $g=6$. Then $H^{0}(N_{C/\mathbb{P}^6}(-2))=(0)$.
%\end{cor}
\textbf{Acknowledgements}\\
This paper arose as an attempt to carry on Sernesi's work on deformations of fibrations. The author is grateful to him for suggesting the problem and providing the technical framework. The author also thanks Ciro Ciliberto and Flaminio Flamini for useful discussions.
\end{section}

$\left.\right.$\\
$\left.\right.$\\
\noindent
Luca Benzo\\
Dipartimento di Matematica\\ 
Università di Roma Tor Vergata\\ 
Via della Ricerca Scientifica 1\\ 
00133 Roma, Italy.\\ 
e-mail benzo@mat.uniroma2.it

\end{document}